%% file: main_v2.1.tex
\documentclass[11pt,reqno]{amsart}
\usepackage[T1]{fontenc}
\usepackage{lmodern,amssymb,mathtools,booktabs,array,longtable}
\usepackage[a4paper,margin=28mm]{geometry}
\usepackage{microtype,aliascnt,enumitem}
\usepackage{tikz}
\usetikzlibrary{arrows.meta,decorations.markings}
\usepackage[colorlinks=true,linkcolor=blue,citecolor=blue,urlcolor=blue]{hyperref}
\usepackage[nameinlink,noabbrev]{cleveref}
\numberwithin{equation}{section}
\newtheorem{theorem}{Theorem}[section]
\newaliascnt{lemma}{theorem}\newtheorem{lemma}[lemma]{Lemma}\aliascntresetthe{lemma}
\newaliascnt{proposition}{theorem}\newtheorem{proposition}[proposition]{Proposition}\aliascntresetthe{proposition}
\newaliascnt{corollary}{theorem}\newtheorem{corollary}[corollary]{Corollary}\aliascntresetthe{corollary}
\newtheorem{conjecture}{Conjecture}[section]
\theoremstyle{definition}
\newaliascnt{definition}{theorem}\aliascntresetthe{definition}
\newaliascnt{example}{theorem}\aliascntresetthe{example}
\theoremstyle{remark}
\newaliascnt{remark}{theorem}\aliascntresetthe{remark}
\newtheorem*{remark*}{Remark}
\crefname{theorem}{Theorem}{Theorems}\Crefname{theorem}{Theorem}{Theorems}
\crefname{lemma}{Lemma}{Lemmas}\Crefname{lemma}{Lemma}{Lemmas}
\crefname{proposition}{Proposition}{Propositions}\Crefname{proposition}{Proposition}{Propositions}
\crefname{corollary}{Corollary}{Corollaries}\Crefname{corollary}{Corollary}{Corollaries}
\crefname{conjecture}{Conjecture}{Conjectures}\Crefname{conjecture}{Conjecture}{Conjectures}
\crefname{definition}{Definition}{Definitions}\Crefname{definition}{Definition}{Definitions}
\crefname{example}{Example}{Examples}\Crefname{example}{Example}{Examples}
\crefname{remark}{Remark}{Remarks}\Crefname{remark}{Remark}{Remarks}
\crefname{section}{Section}{Sections}\Crefname{section}{Section}{Sections}
\crefname{appendix}{Appendix}{Appendices}\Crefname{appendix}{Appendix}{Appendices}
\DeclareMathOperator{\diag}{diag}
\DeclareMathOperator{\tr}{tr}
\DeclareMathOperator{\rank}{rank}
\DeclareMathOperator{\Pf}{Pf}
\DeclareMathOperator{\adj}{adj}

\newcommand{\one}{\mathbf 1}

\newcommand{\ZZ}{\mathbb Z}

\newcommand{\FF}{\mathbb F}
\newcommand{\T}{\mathsf T}

\hypersetup{pdftitle={Extremal arborescences in graph orientations},pdfauthor={Jianing Liu and Zhijun Xu},pdfsubject={Version 2.0: Conjectures 5.1 and 5.4 and sharp bounds for complete bipartite orientations}}
\title[Extremal arborescences in graph orientations]{Extremal arborescences in graph orientations\textsuperscript{*}}
\author[Jianing Liu]{Jianing Liu\textsuperscript{\ensuremath{\dagger}}}
\author[Zhijun Xu]{Zhijun Xu\textsuperscript{\ensuremath{\ddagger}}}
\date{}
\newcommand{\paperfunding}{\textbf{Funding:} Both authors were supported by the Fundamental and Interdisciplinary Disciplines Breakthrough Plan of the Ministry of Education of China (JYB2025XDXM207).}
\newcommand{\paperaffiliation}{School of Mathematical Sciences and LPMC, Nankai University, Tianjin 300071, China.}
\makeatletter
\def\@adminfootnotes{%
  \begingroup
  \renewcommand{\thefootnote}{\fnsymbol{footnote}}%
  \begin{NoHyper}
  \footnotetext[1]{\paperfunding}
  \footnotetext[2]{\paperaffiliation\ Email: \mbox{\href{mailto:liujianing@mail.nankai.edu.cn}{\nolinkurl{liujianing@mail.nankai.edu.cn}}}.}
  \footnotetext[3]{\paperaffiliation\ Email: \mbox{\href{mailto:zhj_xu@mail.nankai.edu.cn}{\nolinkurl{zhj_xu@mail.nankai.edu.cn}}}.}
  \end{NoHyper}
  \endgroup
}
\def\@setauthors{%
  \begingroup
  \trivlist
  \centering\small\@topsep30\p@\relax
  \advance\@topsep by -\baselineskip
  \item\relax
  \author@andify\authors
  \authors
  \endtrivlist
  \endgroup
}
\let\paper@setabstracta\@setabstracta
\def\@setabstracta{%
  \paper@setabstracta
  \begingroup
  \Small
  \leftskip3pc \rightskip3pc
  \par\medskip
  \noindent\textbf{Key words.}\enspace\@keywords.\par
  \endgroup
}
\makeatother
\begin{document}
\begin{abstract}
In this paper, we study the extremal number of arborescences over orientations of a fixed graph. We prove that every alternating orientation of a plane Eulerian multigraph minimizes the rooted arborescence count among its Eulerian orientations. We also prove that every tournament of odd order maximizing the total number of arborescences is regular, and hence Eulerian. These results resolve Conjectures~5.1 and~5.4 of Bandekar et al.~\cite{bandekar2025}. We further establish upper bounds for arborescence counts in orientations of complete bipartite graphs and characterize the orientations attaining them.
\end{abstract}
\maketitle
\enlargethispage{2pt}

\section{Introduction}\label{sec:introduction}
Spanning-tree enumeration connects the combinatorial structure of a graph with the algebra of its Laplacian. Its directed counterpart counts arborescences: spanning directed trees oriented towards a root. The matrix-tree theorem expresses their number as a Laplacian cofactor; see Tutte~\cite{tutte1948} and Chaiken~\cite{chaiken1982}. McKay~\cite{mckay1983} obtained spanning-tree bounds for regular graphs in terms of short cycles. For a fixed underlying graph, we ask which orientations minimize or maximize the number of arborescences.

Let $a_r(D)$ count the arborescences of $D$ rooted at $r$, and let $\tau(D)=\sum_{r\in V(D)}a_r(D)$. The BEST theorem~\cite{best1951} relates these counts in Eulerian digraphs to directed Euler tours through degree-dependent factors. For an Eulerian digraph, the rooted counts have a common value $a(D)$ and $\tau(D)=|V(D)|a(D)$; see \cref{lem:root-independence}. We optimize $a(D)$ over Eulerian orientations and $\tau(D)$ over arbitrary orientations.

Bandekar et al.~\cite{bandekar2025} established lower bounds and equality characterizations for Eulerian orientations of complete graphs, complete bipartite graphs and double graphs, and upper bounds attained by doubly regular tournaments. Their questions include the following two conjectures, which are the main focus of this paper.

\begin{samepage}
\begin{conjecture}[{\cite[Conjecture~5.1 and Remark~5.2]{bandekar2025}}]\label{conj:bandekar-planar}
An alternating orientation of a plane Eulerian graph minimizes the rooted arborescence count among its Eulerian orientations.
\end{conjecture}

\begin{conjecture}[{\cite[Conjecture~5.4]{bandekar2025}}]\label{conj:bandekar-regular}
Every tournament maximizing $\tau$ among all tournaments of the same odd order is regular.
\end{conjecture}
\end{samepage}

\subsection{Alternating orientations of plane graphs}
A plane graph has a fixed embedding. An orientation is alternating if incoming and outgoing arcs alternate in the cyclic order around every vertex. In this order, a loop is recorded twice, once as an incoming arc at its head and once as an outgoing arc at its tail. Such an orientation is Eulerian, since every vertex has equal indegree and outdegree. It always exists when the plane graph is Eulerian. Indeed, every cut of an even-degree graph has even size: summing the degrees on one side counts the internal edges twice and the cut edges once. In particular, the graph has no bridge, so its planar dual has no loop. Each simple cycle of the dual is a Jordan curve, and the original edges it crosses form the cut between the vertices inside and outside that curve. The dual therefore has no odd cycle and is bipartite. The faces of the original graph admit a black-and-white coloring, with different colors on the two sides of every edge. Orient each edge with its black face on the left. At a vertex, consecutive face sectors have opposite colors, so the incident arcs alternate between incoming and outgoing in the cyclic order. At an isolated vertex the condition is vacuous. \Cref{fig:alternating} illustrates this construction.

The following theorem proves \cref{conj:bandekar-planar}, allowing loops, parallel edges and cut vertices.

\begin{theorem}\label{thm:planar-main}
Let $G$ be a finite plane Eulerian multigraph, and let $D_0$ be an alternating orientation of $G$. For every Eulerian orientation $D$ of $G$ and every vertex $r$,
\[
 a_r(D_0)\le a_r(D).
\]
\end{theorem}

The proof strategy is to compare arborescence counts through a forest expansion in skew-adjacency determinants. The Murasugi--Stoimenow lemma~\cite{murasugi2003} gives values in $\{0,1\}$ for an alternating orientation; parity and nonnegativity then show that this orientation minimizes every summand.

\begin{figure}[htb]
\centering

\begin{tikzpicture}[x=12mm,y=12mm,
  vertex/.style={circle,fill=black,inner sep=1.25pt},
  edge/.style={line width=.65pt},
  oriented/.style={edge,postaction={decorate},
    decoration={markings,mark=at position .53 with
      {\arrow{Stealth[length=2mm,width=1.45mm]}}}}]
\begin{scope}
  \coordinate (o) at (0,0);
  \coordinate (a) at (-1.35,1);
  \coordinate (b) at (-1.35,-1);
  \coordinate (c) at (1.35,-1);
  \coordinate (d) at (1.35,1);
  \draw[edge] (o)--(a)--(b)--cycle;
  \draw[edge] (o)--(c)--(d)--cycle;
  \foreach \v in {o,a,b,c,d} \node[vertex] at (\v) {};
  \node at (0,-1.5) {$G$};
\end{scope}
\begin{scope}[xshift=62mm]
  \coordinate (o) at (0,0);
  \coordinate (a) at (-1.35,1);
  \coordinate (b) at (-1.35,-1);
  \coordinate (c) at (1.35,-1);
  \coordinate (d) at (1.35,1);
  \fill[black!12] (o)--(a)--(b)--cycle;
  \fill[black!12] (o)--(c)--(d)--cycle;
  \foreach \s/\t in {o/a,a/b,b/o,o/c,c/d,d/o}
    \draw[oriented] (\s)--(\t);
  \draw[gray,densely dashed,line width=.4pt] (o) circle[radius=3.4mm];
  \foreach \v in {o,a,b,c,d} \node[vertex] at (\v) {};
  \node at (0,-1.5) {$D_0$};
\end{scope}
\end{tikzpicture}

\caption{A plane Eulerian graph and an alternating orientation. In the right-hand drawing, the shaded bounded faces are black, the unbounded face is white, and every edge has its black face on the left. At the central vertex, the incoming and outgoing arcs alternate.}
\label{fig:alternating}
\end{figure}
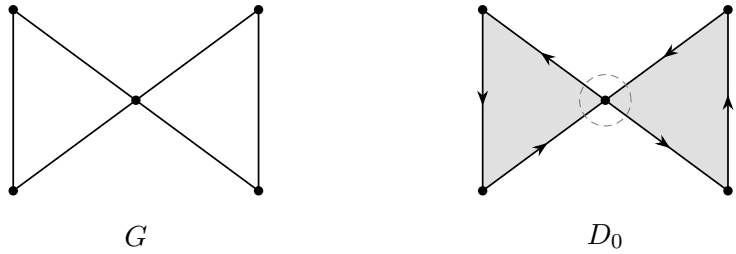

The statement of Conjecture~5.1 in~\cite{bandekar2025} specifies that the underlying graph is Eulerian, without explicitly restricting the orientations being compared. An unrestricted reading admits a simple counterexample: a directed $C_4$ has one arborescence at each root, whereas orienting its edges towards one bipartition gives two sinks and no arborescence. Strong connectivity also gives a different minimum. Embed $K_{2,4}$ as four noncrossing paths $u w_i v$. An alternating orientation directs two paths each way and has $a_r=2$ at every root and $\tau=12$. Directing three paths from $u$ to $v$ and one back gives a strongly connected orientation with $a_u=1$ and $\tau=10$, as follows by counting the outgoing choices at $u$ and $v$.

Remark~5.2 of~\cite{bandekar2025} identifies the intended nontrivial formulation through symmetric edge polytopes. For a plane Eulerian graph $G$ with at least one edge, the facets of the symmetric edge polytope of its dual correspond to planar duals of Eulerian orientations of $G$, and their normalized volumes equal the common rooted arborescence counts. The conjecture asks whether the facet corresponding to an alternating orientation has minimum volume. Accordingly, \cref{conj:bandekar-planar} makes the Eulerian comparison class explicit.

\subsection{Regularity of maximizing tournaments}
A tournament is an orientation of a complete graph. At odd order $n=2d+1$, it is regular when every outdegree equals $d$, equivalently when it is Eulerian. Bandekar et al.~\cite[Theorem~1.6]{bandekar2025} proved that the swirl tournament is, up to isomorphism, the unique regular tournament minimizing the number of arborescences. Its arborescence count provides a lower bound for the maximum over all tournaments of the same order.

The following theorem proves \cref{conj:bandekar-regular}.
\begin{theorem}\label{thm:regular-main}
Let $n$ be a positive odd integer. Every tournament maximizing $\tau(T)$ among all tournaments on $n$ vertices is regular.
\end{theorem}

The proof strategy is to compare arborescence counts under arc reversals. A minimum-outdegree bound controls the inverse matrices in the determinant comparison; the key observation is that any nonregular maximizer would admit a strictly improving reversal of one arc or of a directed path of length two.

\subsection{Complete bipartite graphs}
Problem~5.5 of~\cite{bandekar2025} asks for the orientations of $K_{n,m}$ maximizing $a$ in the Eulerian class and $\tau$ in the unrestricted class. A symmetric determinant identity provides a common approach to both questions: an inequality determined by the column sums of a $0$--$1$ matrix encoding the orientation, together with an exact matrix condition for equality.

For $2\le n\le m$, let
\[
 M_U(n,m)=\max\{\tau(D):D\text{ is an orientation of }K_{n,m}\}.
\]
The first consequence determines the maximum and all maximizing orientations when $m\ge n(n-1)$.
\begin{theorem}\label{thm:intro-bipartite-large}
If $2\le n\le m$ and $m\ge n(n-1)$, then
\[
 M_U(n,m)=2n^{n+m-1}\left(\frac{n-1}{n}\right)^{n(n-1)}.
\]
Up to permutations within the two parts, the unique maximizing orientation has, for each vertex of the smaller part, exactly $n-1$ larger-part vertices whose sole incoming arc comes from that vertex. The remaining $m-n(n-1)$ larger-part vertices are sources.
\end{theorem}

When the average outdegree $d=nm/(n+m)$ is an integer, we prove
\[
 M_U(n,m)\le 2\left(\frac n{n-1}\right)^{n-1}d^{n+m-1}.
\]
Equality is characterized by an incidence design with prescribed block size and pair concurrence. A residual construction from McFarland difference sets~\cite{mcfarland1973} supplies infinitely many equality cases with $n<m<n(n-1)$. For Eulerian orientations, the count reduces to a Gram determinant. Integer common-neighbor counts yield a sharper bound whenever $2(n-1)\nmid m$. We characterize equality and construct infinite attaining families from balanced sign matrices; in particular, we determine the maximum on $K_{n,2n}$ for every $n=2^r\ge4$.

The Eulerian problem also has a geometric interpretation, complementing the facet decompositions of bipartite symmetric edge polytopes studied by K\'alm\'an and T\'othm\'er\'esz~\cite{kalman2022}. For even $n,m$, consider the cographic matroid of $K_{n,m}$, whose bases are the complements of its spanning trees. As recalled in~\cite[Section~1.2]{bandekar2025}, the facets of its symmetric edge polytope correspond to Eulerian orientations, and their normalized volumes are the common rooted arborescence counts. Our Eulerian bounds and equality constructions therefore give volume bounds and explicit attaining facets for this family, without requiring a planar embedding.

The paper is organized as follows. \Cref{sec:prelim} collects the matrix-tree identities, determinant estimates and results used in the proofs. Sections~\ref{sec:alternating} and~\ref{sec:tournaments} prove Theorems~\ref{thm:planar-main} and~\ref{thm:regular-main}, respectively. \Cref{sec:problem55} establishes the complete bipartite bounds and their equality cases. \Cref{sec:problem53} discusses maximizing regular tournaments. Proofs of the spectral product bound, the Eulerian sharpness constructions and the equivalent formulations of near double regularity are collected in the appendices.

\section{Preliminaries}\label{sec:prelim}
A digraph is called connected if its underlying undirected graph is connected.
Unless stated otherwise, all graphs and digraphs in this paper are finite
and connected; loops and parallel edges are allowed. A disconnected graph
has no spanning tree, and a digraph with disconnected underlying graph
has no spanning arborescence, so the corresponding counts are zero.
A loop contributes two to the undirected degree and one to each directed
degree. It belongs to no spanning tree. A plane graph has a fixed
embedding, which determines the cyclic order of incident arcs in any
orientation. A loop is recorded twice in this order, once at its head
as incoming and once at its tail as outgoing.
A graph is Eulerian when all its degrees are even; a
digraph is Eulerian when its indegree and outdegree are equal at each vertex.
Under the standing connectivity assumption, these conditions are equivalent
to the existence of an undirected or directed Euler circuit, respectively.

For a digraph $D$ and a vertex $r$, let $a_r(D)$ count the spanning
arborescences directed towards $r$, and let
$\tau(D)=\sum_{r\in V(D)}a_r(D)$. A single vertex, with no selected edges,
counts as one arborescence. For an Eulerian digraph, let $a(D)$ denote
the common rooted count, so that $\tau(D)=|V(D)|a(D)$.

Let $A_D$ be the adjacency matrix of the nonloop arcs and let
$M_D=A_D-A_D^{\T}$ be its skew adjacency matrix. Thus $(A_D)_{uv}$
counts arcs $u\to v$ for $u\ne v$, and $(A_D)_{uu}=0$.
Let $\one$ be the all-ones vector, $I$ the identity matrix and
$J=\one\one^{\T}$, with dimensions clear from the context. The
out-Laplacian is
\[
 L_D=\diag(A_D\one)-A_D.
\]
For an undirected multigraph $G$, the off-diagonal entry of its Laplacian
$L_G$ is minus the edge multiplicity, and its diagonal entry counts
incident nonloop edges. Loops therefore contribute zero to $L_D$, $L_G$
and $M_D$. Eulerian balance gives
\[
 L_D=\frac12(L_G-M_D)
\]
for every Eulerian orientation $D$ of $G$. For a square matrix $C$,
let $C^r$ be the principal submatrix obtained by deleting row and
column $r$. Its adjugate $\adj C$ is the transpose of its cofactor
matrix: the $(i,j)$ entry is $(-1)^{i+j}$ times the determinant obtained
by deleting row $j$ and column $i$ from $C$. In particular,
$C\adj C=(\adj C)C=(\det C)I$. Transposition is denoted by $C^{\T}$. Vector norms are
Euclidean, and $\|C\|_F^2=\sum_{i,j}|C_{ij}|^2$ denotes the squared
Frobenius norm.

For a linear map $C:V\to V$ on a finite-dimensional Euclidean space
and an invariant subspace $\mathcal H\subseteq V$, meaning
$C\mathcal H\subseteq\mathcal H$, the restriction
$C|_{\mathcal H}:\mathcal H\to\mathcal H$ is defined by
$C|_{\mathcal H}(x)=Cx$ for $x\in\mathcal H$. We use the same notation
when $C$ is represented by a matrix. The determinant
$\det(C|_{\mathcal H})$ is computed in any basis of $\mathcal H$;
it is independent of that basis. For $\mathcal H=\{0\}$, this
determinant is $1$.
If $P$ is the orthogonal projection onto a subspace $\mathcal H$,
the compression of $C$ to $\mathcal H$ is $(PCP)|_{\mathcal H}$;
this definition does not require $C\mathcal H\subseteq\mathcal H$.

\subsection{Matrix-tree and determinant identities}
The directed and undirected matrix-tree theorems will be used in the
following convention; see \cite[Theorems~2.4--2.5]{bandekar2025}.
\begin{lemma}[Matrix-tree theorem]\label{lem:matrix-tree}
For a digraph $D$ and a vertex $r$,
\[
 a_r(D)=\det L_D^r.
\]
For an undirected graph $G$, $\det L_G^r$ is its number of spanning
trees. The same identities hold with edge or arc weights, with each
tree counted by the product of its weights. Moreover,
\[
 \operatorname{adj}L_D=\one\,(a_r(D))_{r\in V(D)}^{\T},
\]
where the vector on the right lists the rooted counts.
\end{lemma}
For nonnegative integer weights, the weighted identity follows by
replacing an arc by the corresponding number of copies; polynomial
extension gives the general weighted form. To see the adjugate identity,
if $\rank L_D=n-1$, the identity
$L_D\one=0$ makes every column of its adjugate constant, and its
diagonal entries are the principal cofactors. If the rank is smaller,
both sides are zero.

\begin{lemma}\label{lem:root-independence}
An Eulerian digraph is strongly connected, and its rooted
arborescence counts are independent of the root. If it has $n$ vertices
and $\mathcal H=\one^\perp\subseteq\mathbb R^n$, then
\[
 \det(L_D|_{\mathcal H})=na(D).
\]
\end{lemma}
\begin{proof}
A source in the condensation of $D$ into strongly connected components
has no incoming arcs. Summing the balance equations over that component
shows that it has no outgoing arcs either. Connectedness forces the
condensation to have one vertex, so $D$ is strongly connected.

Let $n=|V(D)|$. By the matrix-tree theorem, strong connectivity gives
$\rank L_D=n-1$. Moreover, $L_D\one=0$ and $\one^{\T}L_D=0$.
Its left and right kernels are therefore spanned by $\one$, so its
adjugate is a scalar multiple of $\one\one^{\T}$. Its principal
cofactors coincide, proving root independence.
In an orthonormal basis beginning with $n^{-1/2}\one$, the matrix
is $0\oplus (L_D|_{\mathcal H})$. The coefficient of $t$ in
$\det(tI+L_D)$ is both $\det(L_D|_{\mathcal H})$ and the sum of its
$n$ principal cofactors. This proves the determinant identity.
The one-vertex case uses the empty determinant convention.
\end{proof}

The standard determinant identities below may also be found in
\cite{hornjohnson2013}. The following scaled form of
\cite[Lemma~2.8]{bandekar2025} relates a shifted Laplacian determinant to the total arborescence count.
\begin{lemma}\label{lem:tour-shift}
For a digraph $D$ of order $n\ge1$ and real numbers $s,\alpha$,
\[
 \det(sL_D+\alpha J)=n\alpha s^{n-1}\tau(D).
\]
\end{lemma}
\begin{proof}
In an orthonormal basis with first vector $n^{-1/2}\one$, the identity
$L_D\one=0$ gives
\[
 L_D=\begin{pmatrix}0&z^{\T}\\0&C\end{pmatrix},
 \qquad J=\begin{pmatrix}n&0\\0&0\end{pmatrix}.
\]
The coefficient of $t$ in $\det(tI+L_D)=t\det(tI+C)$ is $\det C$.
Expansion by principal minors and \cref{lem:matrix-tree} identify the
same coefficient as $\sum_r a_r(D)=\tau(D)$. Taking the block
determinant of $sL_D+\alpha J$ proves the formula, also for $s=0$.
At $n=1$, the lower block is empty and $s^{n-1}=1$.
\end{proof}

\begin{lemma}\label{lem:regular-determinant}
Let $T$ be a regular tournament of odd order $n$, and let
$\Phi_{M_T}(x)=\det(xI-M_T)$. Then
\begin{equation}\label{eq:regular-arborescence}
 a(T)=\frac{\Phi_{M_T}(n)}{n^2 2^{n-1}}.
\end{equation}
\end{lemma}
\begin{proof}
Since $A_T+A_T^{\T}=J-I$ and every outdegree is $(n-1)/2$,
\[
 2L_T=nI-J-M_T.
\]
Thus $nI-M_T=2L_T+J$, and \cref{lem:tour-shift} with $s=2$,
$\alpha=1$ gives
$\Phi_{M_T}(n)=n2^{n-1}\tau(T)=n^2 2^{n-1}a(T)$.
\end{proof}

\begin{lemma}[Sylvester and Schur identities]\label{lem:sylvester}
If $U$ is $p\times q$ and $V$ is $q\times p$, then
\[
 \det(I_p+UV)=\det(I_q+VU).
\]
For a block matrix whose square lower-right block $C$ is invertible,
\[
 \det\begin{pmatrix}E&F\\G&C\end{pmatrix}
 =\det C\,\det(E-FC^{-1}G).
\]
For a square matrix $C$ and column vectors $u,v$ of matching size,
\[
 \det(C+uv^{\T})=\det C+v^{\T}(\operatorname{adj}C)u.
\]
When $C$ is invertible, this is
$\det C\,(1+v^{\T}C^{-1}u)$.
\end{lemma}
\begin{proof}
Multiplication on the left by
$\left(\begin{smallmatrix}I&-FC^{-1}\\0&I\end{smallmatrix}\right)$
proves the Schur identity. Applying it to
$\left(\begin{smallmatrix}I_p&U\\-V&I_q\end{smallmatrix}\right)$,
and then eliminating its upper-left block instead, gives
Sylvester's identity. The rank-one formula for invertible $C$ follows by taking $q=1$;
polynomial extension gives the same formula for singular
$C$ as well.
\end{proof}

\begin{lemma}[Hadamard's determinant inequality, \cite{hornjohnson2013}]\label{lem:hadamard}
Let $C$ be a real $q\times q$ matrix with row vectors $r_1,\ldots,r_q$.
Then
\[
 |\det C|\le\prod_{i=1}^q\|r_i\|.
\]
If every row is nonzero, equality holds if and only if the rows are
pairwise orthogonal.
\end{lemma}

The following inverse estimates control determinant changes under arc reversals.

\begin{lemma}\label{lem:tour-cayley}
Let $C=\Delta-\Omega$, where $\Delta$ is a positive diagonal matrix with
entries at least $\mu>0$, and $\Omega$ is real skew-symmetric. Then $C$ is
invertible and $\det C>0$. For real vectors $p,q$ with disjoint supports,
\[
 |p^{\mathsf T}C^{-1}q|
 \le\frac12
 \sqrt{(p^{\mathsf T}\Delta^{-1}p)(q^{\mathsf T}\Delta^{-1}q)}
 \le\frac{\|p\|\|q\|}{2\mu}.
\]
For every real vector $z$,
\[
 0\le z^{\mathsf T}C^{-1}z
 \le z^{\mathsf T}\Delta^{-1}z
 \le\frac{\|z\|^2}{\mu}.
\]
\end{lemma}

\begin{proof}
For nonzero real $y$ and $0\le t\le1$,
\[
 y^{\mathsf T}(\Delta-t\Omega)y=y^{\mathsf T}\Delta y>0.
\]
Hence $\Delta-t\Omega$ is nonsingular. Continuity of its determinant
from $t=0$ gives $\det C>0$.

Let
\[
 H=\Delta^{-1/2}\Omega\Delta^{-1/2},\qquad
 W=(I+H)(I-H)^{-1}.
\]
Since $H$ is skew-symmetric, its Cayley transform $W$ is orthogonal:
$W^{\mathsf T}=(I+H)^{-1}(I-H)$, and the commuting factors $I-H,I+H$
give $W^{\mathsf T}W=I$. Moreover, $I+W=2(I-H)^{-1}$, so
\[
 C^{-1}
 =\frac12\Delta^{-1/2}(I+W)\Delta^{-1/2}.
\]
If $p,q$ have disjoint supports, then $p^{\mathsf T}\Delta^{-1}q=0$.
Using this inverse formula, orthogonality and Cauchy--Schwarz proves
the disjoint-support estimate. For $y=\Delta^{-1/2}z$,
\[
 0\le z^{\mathsf T}C^{-1}z
 =\tfrac12(\|y\|^2+y^{\mathsf T}Wy)
 \le\|y\|^2=z^{\mathsf T}\Delta^{-1}z,
\]
since $|y^{\mathsf T}Wy|\le\|y\|^2$. Finally, the lower bound on the
diagonal entries of $\Delta$ gives the remaining estimates.
\end{proof}

\subsection{Tournament bounds and alternating tree polynomials}
For odd $n=2d+1$, the swirl tournament $\mathrm{SW}_n$ has vertex set
$\mathbb Z/n\mathbb Z$ and arcs $i\to i+j$ for $1\le j\le d$.
The following result of Bandekar et al.\ \cite[Theorem~1.6]{bandekar2025}
provides the explicit comparison value used in \cref{sec:tournaments}.
\begin{lemma}
\label{lem:tournament-minima}
For every regular tournament $T$ of odd order $n\ge1$,
\[
 a(T)\ge\frac{((n+1)/2)^n+((n-1)/2)^n}{n^2},
\]
with equality precisely when $T\cong\mathrm{SW}_n$.
\end{lemma}
In the total-count normalization this gives
\[
 \tau(\mathrm{SW}_n)=\frac{((n+1)/2)^n+((n-1)/2)^n}{n}.
\]

The following upper bounds are due to Bandekar et al.~\cite[Theorems~1.9 and~1.11]{bandekar2025}.
\begin{lemma}\label{lem:bandekar-upper}
Let $D$ be a loopless simple digraph with $N\ge2$ vertices, $e$ arcs
and outdegrees $d_1,\ldots,d_N$. Then
\[
 \tau(D)\le
 \left(\frac{\sum_{i=1}^N d_i^2+e}{N-1}\right)^{(N-1)/2}.
\]
In particular, for a regular tournament $T$ of odd order $n\ge3$,
\[
 \tau(T)\le U_n:=\left(\frac{n(n+1)}4\right)^{(n-1)/2},
 \qquad a(T)\le\frac{U_n}{n}.
\]
\end{lemma}

For an Eulerian orientation $D$ of a multigraph $G$, let
$i_D(T,r)$ be the number of edges of an undirected spanning tree $T$
whose direction must be reversed to orient $T$ towards $r$, and let
\[
 P_{D,r}(t)=\sum_{T\text{ spanning tree of }G}t^{i_D(T,r)}.
\]
Replacing each nonloop arc by a forward copy of weight $1$ and a reverse
copy of weight $t$ gives, by the weighted matrix-tree theorem,
\[
 P_{D,r}(t)=\det\!\left((L_D+tL_D^{\T})^r\right).
\]
Eulerian balance is needed here to identify the reverse-arc Laplacian
with $L_D^{\T}$. Parallel copies remain distinguished, and loops do not
enter the spanning trees or the weighted Laplacian.

\begin{lemma}[Murasugi--Stoimenow, \cite{murasugi2003}]\label{lem:ms-polynomial}
Let $D_0$ be an alternating orientation of a plane Eulerian
multigraph. For every vertex $r$,
\[
 P_{D_0,r}(-1)\in\{0,1\}.
\]
More precisely, its value is congruent modulo two to $P_{D_0,r}(1)$,
the number of spanning trees of the underlying graph.
\end{lemma}
This is \cite[Theorem~3(7)]{murasugi2003}, which applies to plane Eulerian multigraphs, including loops, parallel edges and cut vertices.

\section{Alternating orientations}\label{sec:alternating}
We prove \cref{thm:planar-main} by expressing $2^{|V(G)|-1}a_r(D)$
as a sum of skew-adjacency cofactors of graphs obtained by contracting
spanning forests. Each summand is a nonnegative integer.
For an alternating orientation, the contracted orientation remains
alternating, and the Murasugi--Stoimenow lemma makes each such cofactor
equal to $0$ or $1$. Reorienting the graph preserves the parity of every
cofactor. The alternating orientation therefore minimizes each summand,
and hence the arborescence count.

\subsection{A forest expansion}
A spanning forest $F$ of $G$ contains every vertex and has no cycles;
it may be disconnected, and isolated vertices are allowed. Let $c(F)$ denote its number of components and $r_F$ the
component containing $r$. Contract every edge of $F$, retaining all
remaining edge copies and any loops created by contraction. The induced
orientation $D/F$ of $G/F$ is Eulerian: summing the vertex balance
equations over a component cancels its internal arcs and leaves equal
incoming and outgoing boundary counts.

For an even-order skew-symmetric matrix $S$, its Pfaffian satisfies
$\det S=\Pf(S)^2$; see \cite{ledermann1993}. Since the Pfaffian is a
polynomial with integer coefficients in the entries of $S$, this
determinant is a nonnegative integer square when $S$ is integral.
We take the determinant and Pfaffian of the empty matrix to be $1$.

\begin{lemma}[Forest expansion]\label{alt-lem-expansion}
Let $D$ be an Eulerian orientation of a multigraph $G$ of
order $n\ge1$. Then, for every vertex $r$,
\begin{equation}\label{alt-eq-expansion}
 2^{n-1}a_r(D)=
 \sum_{\substack{F\text{ spanning forest of }G\\c(F)\text{ odd}}}
 \det M_{D/F}^{r_F}
 =\sum_{\substack{F\text{ spanning forest of }G\\c(F)\text{ odd}}}
 \Pf\bigl(M_{D/F}^{r_F}\bigr)^2.
\end{equation}
\end{lemma}
\begin{proof}
For each nonloop edge $e$, choose a sign for its incidence vector and
delete the coordinate of $r$. Denote the resulting vector by
$b_e\in\ZZ^{n-1}$. Then
\[
 L_G^r=\sum_e b_eb_e^{\T}.
\]
For an arbitrary $(n-1)\times(n-1)$ matrix $N$, consider
\[
 P(\mathbf w)=\det\left(N+\sum_e w_e b_eb_e^{\T}\right).
\]
We identify its coefficients by the forest structure of the selected
edges, and then specialize to $N=-M_D^r/2$ to recover the arborescence
count.

Let $E_*$ be the set of nonloop edges, with parallel copies represented
as separate elements. Expanding the determinant by its columns gives
\[
 P(\mathbf w)=\sum_{S\subseteq E_*}\gamma_S
                   \prod_{e\in S}w_e.
\]
Indeed, every determinant term involving $w_e^2$ has two columns
proportional to $b_e$, so its determinant is $0$.
More generally, $\gamma_S=0$ whenever the vectors
$(b_e)_{e\in S}$ are linearly dependent. If $S$ contains a cycle
$e_1,\ldots,e_k$, traversing the cycle gives signs
$\varepsilon_i\in\{1,-1\}$ such that
\[
 \varepsilon_1b_{e_1}+\cdots+\varepsilon_kb_{e_k}=0.
\]
This includes a cycle formed by two parallel edges. Conversely, suppose
$S$ is a forest and $\sum_{e\in S}x_eb_e=0$. A leaf vertex distinct from
$r$ has a coordinate occurring in exactly one vector, so the coefficient
of its incident edge is zero. Removing that edge and repeating proves
$x_e=0$ for every $e\in S$. Thus the contributing edge sets are forests.

We now compute the coefficient of a forest $F$ with $h$ edges and
$c=n-h$ components. In each
component $C\ne r_F$, choose a vertex $v_C$, and adjoin an auxiliary
edge joining $r$ to $v_C$. Each auxiliary edge joins the root component
to a different component of $F$, so together they form a spanning tree.
These edges are used to choose coordinates and need not belong to $G$.
Order the forest edges as $e_1,\ldots,e_h$ and the nonroot components as
$C_1,\ldots,C_{c-1}$. The reduced incidence matrix of this tree, with
suitable column signs, is
\[
 \begin{gathered}
 T=[\,B_F\mid E_F\,],\\
 B_F=[\,b_{e_1}\ \cdots\ b_{e_h}\,],\qquad
 E_F=[\,e_{v_{C_1}}\ \cdots\ e_{v_{C_{c-1}}}\,].
 \end{gathered}
\]
Here $e_v$ denotes the standard coordinate vector at $v$. A nonroot
leaf of the tree gives a row of $T$ with a single nonzero entry, equal
to $1$ or $-1$. Expansion along that row reduces $\det T$ to the reduced
incidence determinant of a smaller tree. Induction, starting with the
empty matrix, gives $\det T=\pm1$.

Let $Q=T^{-1}$, and let $U$ have rows $\one_{C_1}^{\T},\ldots,
\one_{C_{c-1}}^{\T}$, where the indicator vectors have coordinates
indexed by $V(G)\setminus\{r\}$. Every forest edge has both endpoints
in one component, whereas $v_{C_j}$ belongs to $C_j$. Consequently,
\[
 UB_F=0,\qquad UE_F=I_{c-1},\qquad
 U=[\,0\mid I_{c-1}\,]Q.
\]
The last equality follows by multiplying $UT=[\,0\mid I_{c-1}\,]$
by $Q$. It identifies the last $c-1$ rows of $Q$ with component
indicators and explains how the change of coordinates separates the
forest edges from the component sums.

Let every variable outside $F$ be zero. Since $(\det Q)^2=1$,
congruence by $Q$ leaves the determinant unchanged and yields
\[
 \det\left(N+\sum_{e\in F}w_e b_eb_e^{\T}\right)
 =\det\left(QNQ^{\T}
 +\diag(w_{e_1},\ldots,w_{e_h},0,\ldots,0)\right).
\]
The coefficient of $\prod_{e\in F}w_e$ is obtained by selecting all
first $h$ diagonal entries. It is therefore the determinant of the
lower-right $(c-1)\times(c-1)$ principal block of $QNQ^{\T}$,
with positive sign.

To identify this block with the quotient matrix, take $N=-M_D^r/2$.
The component-indicator rows sum the matrix entries over each pair
of components, so the remaining block has entries
\[
 -\frac12\sum_{u\in C}\sum_{v\in C'}(M_D)_{uv}
 =-\frac12(M_{D/F})_{CC'}.
\]
Here the double sum is the number of arcs directed from $C$ to $C'$
minus the number directed from $C'$ to $C$, with parallel arcs counted
separately. For $C=C'$, the sum is zero by skew-symmetry. Any remaining
arc within a component becomes a loop after contraction. Hence the block
is $-M_{D/F}^{r_F}/2$.

By the Eulerian identity $2L_D=L_G-M_D$ and \Cref{lem:matrix-tree},
evaluating $P(\mathbf w)$ at $w_e=1/2$ gives $a_r(D)$.
The contribution indexed by $F$ is consequently
\[
 2^{-h}\det\bigl(-M_{D/F}^{r_F}/2\bigr).
\]
For even $c$, the skew-symmetric matrix has odd order, so
$\det(-M_{D/F}^{r_F}/2)=0$. For odd $c$, the reduced order $c-1$ is even,
and $h+c-1=n-1$, so the contribution is
$2^{-(n-1)}\det M_{D/F}^{r_F}$. Summing proves the first equality in
\eqref{alt-eq-expansion}; the second is the Pfaffian identity stated
above.
\end{proof}

\subsection{Alternating skew cofactors}
Forest contraction preserves the alternating plane orientation. The Murasugi--Stoimenow evaluation then places every alternating quotient cofactor in $\{0,1\}$, providing the comparison needed in \eqref{alt-eq-expansion}.

\begin{lemma}\label{alt-lem-contraction}
If $D_0$ is an alternating plane orientation and $F$ is a forest, then
$D_0/F$ is alternating in the contracted embedding.
\end{lemma}
\begin{proof}
Contract a nonloop arc $u\to v$. At $u$, remove this arc from the
cyclic order of incident arcs and read the remaining list starting
immediately after its former position. The list alternates and starts
and ends with arcs entering $u$. At $v$, removing the same arc from
the cyclic order similarly leaves a list starting and ending with arcs
leaving $v$. Both lists are nonempty, since the original endpoint
degrees are positive and even. Plane contraction joins the two lists
in their inherited cyclic orders. At each join, an incoming arc is
adjacent to an outgoing arc, so the combined order alternates.
Any other arc joining $u$ and $v$ becomes a loop and retains its two
occurrences in this order, one incoming and one outgoing.

Contract the edges of $F$ successively. Each remains a nonloop until
its contraction, since otherwise $F$ would contain a cycle. The
preceding argument therefore applies at every step.
\end{proof}

\begin{lemma}\label{alt-lem-unit}
Let $D_0$ be an alternating orientation of a plane Eulerian
multigraph of odd order. Then, for every vertex $r$,
\[
 \det M_{D_0}^r\in\{0,1\}.
\]
\end{lemma}
\begin{proof}
Apply \cref{lem:ms-polynomial} to the spanning-tree polynomial
$P_{D_0,r}$. By its weighted matrix-tree identity,
\[
 P_{D_0,r}(-1)
 =\det\!\left((L_{D_0}-L_{D_0}^{\T})^r\right)
 =\det(-M_{D_0}^r).
\]
The reduced order is even, so the final determinant is
$\det M_{D_0}^r$. The empty determinant covers the one-vertex case.
\end{proof}

\begin{proof}[Proof of \Cref{thm:planar-main}]
Fix a spanning forest $F$ with
$c(F)$ odd. The quotient $G/F$ is plane, and
\cref{alt-lem-contraction,alt-lem-unit} give
\[
 \det M_{D_0/F}^{r_F}\in\{0,1\}.
\]
Reversing an edge changes its net skew-adjacency contribution by an
even integer. The same holds after aggregation over forest components.
Consequently,
\[
 M_{D/F}^{r_F}\equiv M_{D_0/F}^{r_F}\pmod2.
\]
Their determinants have the same parity, since a determinant is an
integer polynomial in its entries. Moreover,
$\det M_{D/F}^{r_F}$ is a nonnegative integer Pfaffian square.
If the alternating cofactor is zero, nonnegativity gives the required
comparison. If it is one, the other cofactor is odd and nonnegative,
and therefore at least one. In both cases,
\[
 \det M_{D_0/F}^{r_F}\le\det M_{D/F}^{r_F}.
\]
Summing over $F$ in \eqref{alt-eq-expansion} proves the theorem.
On one vertex there is one arborescence, independently of loops,
in agreement with the empty-determinant convention.
\end{proof}

\section{Tournament maximizers}\label{sec:tournaments}

The proof of \Cref{thm:regular-main} begins with a lower bound on the
minimum outdegree of a maximizing tournament. This bound controls the
inverse matrices that occur when an arc is reversed. The condition that
no single reversal increases the count restricts the degrees to three
consecutive values. Reversing a directed path of length two then
removes the remaining imbalance.

\subsection{Degree bounds}

Let $T$ be a tournament on $\{1,\ldots,n\}$, and let $d_i=d_T^+(i)$.
Let
\[
 x_i=2d_i+1,\qquad X=\diag(x_1,\ldots,x_n).
\]
Since $A_T+A_T^{\mathsf T}=J-I$, we have
\[
 \mathcal B(T):=X-M_T=2L_T+J.
\]
By \Cref{lem:tour-shift},
\[
 \Psi(T):=\det \mathcal B(T)=n2^{n-1}\tau(T).
\]
Thus maximizing $\Psi(T)$ is equivalent to maximizing $\tau(T)$.
The numbers $x_i$ are positive odd integers and satisfy
\[
 \sum_{i=1}^n x_i
 =2\sum_{i=1}^n d_i+n=n^2.
\]
For odd $n$, regularity is equivalent to $x_i=n$ for every $i$.

\begin{proposition}\label{prop:tour-upper}
Every tournament $T$ of order $n\ge2$ satisfies
\[
 \tau(T)\le
 \frac{\sqrt{\prod_{i=1}^n(x_i^2+n)}}{2^{n-1}\sqrt{n(n+1)}}
 \le U_n:=\left(\frac{n(n+1)}4\right)^{(n-1)/2}.
\]
\end{proposition}

\begin{proof}
Let $\beta=\sqrt{(n+1)/n}$ and $\mathcal B_\beta=2L_T+\beta J$.
Since $L_T\one=0$, the inner product of its $i$th row with the
constant row of $\beta J$ is $0$. The squared Euclidean norm of the
$i$th row of $L_T$ is $d_i^2+d_i$, so
\[
 \|(\mathcal B_\beta)_{i,*}\|^2
 =4(d_i^2+d_i)+n\beta^2=x_i^2+n.
\]
\Cref{lem:hadamard,lem:tour-shift} give
\[
 n\beta\,2^{n-1}\tau(T)
 =\det \mathcal B_\beta
 \le\sqrt{\prod_{i=1}^n(x_i^2+n)},
\]
which proves the first inequality.

For $x\ge1$, let
\[
 g_n(x)=\log\frac{x^2+n}{n(n+1)}-\frac{2(x-n)}{n+1}.
\]
We have
\[
 g_n'(x)=-\frac{2(x-1)(x-n)}{(n+1)(x^2+n)}.
\]
Thus $g_n$ increases on $[1,n]$ and decreases on $[n,\infty)$,
with unique maximum $g_n(n)=0$. In particular,
\[
 \log(x^2+n)
 \le\log(n(n+1))+\frac{2(x-n)}{n+1}.
\]
Summing over the vertices and using $\sum_i x_i=n^2$, we obtain
\[
 \prod_{i=1}^n(x_i^2+n)\le[n(n+1)]^n,
\]
proving the second inequality. More precisely, the first bound yields
\begin{equation}\label{eq:tour-penalty-bound}
 \frac{\tau(T)}{U_n}
 \le\exp\left(\frac12\sum_{i=1}^n g_n(x_i)\right).
\end{equation}
The nonpositive terms $g_n(x_i)$ quantify the loss in the bound when the
degrees depart from the regular value.
\end{proof}

The constant $U_n$ is also the total-count bound for regular tournaments
in \Cref{lem:bandekar-upper}. For an arbitrary tournament, the general
degree bound in that lemma reads
\[
 \tau(T)\le
 \left(\frac{n(n+1)}4+
       \frac1{n-1}\sum_{i=1}^n
       \left(d_i-\frac{n-1}{2}\right)^2\right)^{(n-1)/2}.
\]
Its right side equals $U_n$ in the regular case and increases with the
squared degree deviations. The product bound in \Cref{prop:tour-upper}
gives $U_n$ for every tournament and retains the individual degree
factors, which yield the nonpositive correction in
\eqref{eq:tour-penalty-bound}.

\subsection{A minimum-degree bound}

We compare \eqref{eq:tour-penalty-bound} with the count of the swirl
tournament $\mathrm{SW}_n$ from \Cref{lem:tournament-minima}.

\begin{lemma}\label{lem:tour-minimum}
Let $n\ge5$ be odd. If $T$ maximizes $\tau(T)$ among tournaments of order $n$,
then
\[
 \min_i x_i>\frac{n-2}{2}.
\]
Equivalently, its minimum outdegree is at least $\lfloor n/4\rfloor$.
\end{lemma}

\begin{proof}
We show that a tournament violating the asserted degree bound has
$\tau(T)<(9/10)U_n$, whereas $\tau(\mathrm{SW}_n)>(9/10)U_n$.
Suppose some $x_i\le t=(n-2)/2$. Since $g_n$ is nonpositive on
$[1,\infty)$ and increases on $[1,n]$, \eqref{eq:tour-penalty-bound}
implies
\begin{align*}
 \frac{\tau(T)}{U_n}
 &\le\exp\bigl(g_n(t)/2\bigr)\\
 &=\sqrt{\frac{n^2+4}{4n(n+1)}}
   \exp\left(\frac{n+2}{2(n+1)}\right)
 \le\frac12 e^{7/12}<\frac9{10}.
\end{align*}
Indeed, $n^2+4\le n(n+1)$ for $n\ge4$, and
$(n+2)/(2(n+1))\le7/12$ for $n\ge5$. The last strict inequality
uses $e^{7/12}\approx1.79200<9/5$.

On the other hand, \Cref{lem:tournament-minima} gives
\[
 \frac{\tau(\mathrm{SW}_n)}{U_n}
 =\frac12\left[
 (1+1/n)^{(n+1)/2}
 +\frac{(1-1/n)^n}{(1+1/n)^{(n-1)/2}}
 \right].
\]
The first three binomial terms yield
\[
 (1+1/n)^{(n+1)/2}
 \ge1+\frac{n+1}{2n}+\frac{n^2-1}{8n^2}
 =\frac{13}{8}+\frac1{2n}-\frac1{8n^2}
 >\frac{13}{8}.
\]
Also $(1-1/n)^n\ge8/27$ for $n\ge3$: the derivative of
$x\log(1-1/x)$ is
$\log(1-1/x)+1/(x-1)>0$, by
$-\log(1-u)<u/(1-u)$ for $0<u<1$.
Finally,
\[
 (1+1/n)^{(n-1)/2}<\sqrt e<\frac53.
\]
The first inequality follows from $\log(1+s)<s$, and the second from
$e\approx2.71828<25/9$. Consequently,
\[
 \frac{\tau(\mathrm{SW}_n)}{U_n}
 >\frac{13}{16}+\frac4{45}
 =\frac{649}{720}>\frac9{10}.
\]
Thus $\tau(T)<\tau(\mathrm{SW}_n)$, contrary to maximality.
\end{proof}

\subsection{Arc reversals}

The inverse estimates in \Cref{lem:tour-cayley} control the determinant change under an arc reversal.

When taking a principal submatrix of $\mathcal B(T)$, we retain the
diagonal entries $x_i=2d_T^+(i)+1$, computed in the full tournament $T$.
The same convention applies to each tournament obtained by an arc
reversal.

\begin{lemma}\label{lem:tour-arc}
Let $u\to v$ be an arc of a tournament $T$ of order $n\ge3$. Order $u,v$
first and let
\[
 \mathcal B(T)=\begin{pmatrix}
 x_u&-1&p^{\mathsf T}\\
 1&x_v&q^{\mathsf T}\\
 -p&-q&C
 \end{pmatrix},
 \qquad
 \eta=(q-p)^{\mathsf T}C^{-1}(q+p).
\]
If $T^{uv}$ is obtained by reversing $u\to v$, then
\begin{equation}\label{eq:tour-arc}
 \Psi(T^{uv})-\Psi(T)
 =2\det C\,(x_u-x_v-2-\eta).
\end{equation}
If every $x_i\ge\mu>0$, then
\[
 |\eta|\le\frac{n-2}{\mu}.
\]
\end{lemma}

\begin{proof}
The matrix $C$ is a positive diagonal matrix with entries $x_i\ge1$
minus a real skew-symmetric matrix. Thus \Cref{lem:tour-cayley}, with
$\mu=1$, gives $\det C>0$. By the Schur identity in
\Cref{lem:sylvester}, let $H=(h_{ij})_{i,j=1}^2$ be its Schur
complement in $\mathcal B(T)$. Then $\Psi(T)=\det C\,\det H$, where
\[
 h_{11}=x_u+p^{\mathsf T}C^{-1}p,
 \qquad h_{22}=x_v+q^{\mathsf T}C^{-1}q,
\]
\[
 h_{12}=-1+p^{\mathsf T}C^{-1}q,
 \qquad h_{21}=1+q^{\mathsf T}C^{-1}p.
\]
Reversing the arc changes $x_u,x_v$ to $x_u-2,x_v+2$, respectively,
and changes the off-diagonal entries $-1,1$ to $1,-1$. Thus the Schur
complement changes by
\[
 \begin{pmatrix}-2&2\\-2&2\end{pmatrix}.
\]
Its determinant changes by
\[
 2(h_{11}-h_{22}+h_{12}-h_{21})
 =2(x_u-x_v-2-\eta).
\]
Multiplying by $\det C$ proves \eqref{eq:tour-arc}.

Since $p_i,q_i\in\{-1,1\}$, at each coordinate one of $q-p,q+p$ is
zero. These vectors therefore have disjoint supports. If their support
sizes are $a,b$, then $a+b=n-2$, and
\[
 \|q-p\|\,\|q+p\|=4\sqrt{ab}\le2(a+b)=2(n-2).
\]
Applying the disjoint-support estimate in \Cref{lem:tour-cayley} gives
the asserted bound on $\eta$.
\end{proof}

\begin{corollary}\label{cor:tour-arc-improvement}
Let $T$ be a tournament of order $n\ge3$. If $\min_i x_i>(n-2)/2$, reversing an arc $u\to v$ with
$d_u\ge d_v+2$ strictly increases $\Psi$.
If every $x_i\ge n-2$, reversing an arc with $d_u-d_v=1$ changes
$\Psi$ by at least $-2\det C$, whereas reversing an arc with
$d_u-d_v=2$ changes it by at least $2\det C$.
Here $C$ is the principal matrix outside the endpoints, as in
\Cref{lem:tour-arc}.
\end{corollary}

\begin{proof}
Under the first hypothesis, $|\eta|<2$, while $x_u-x_v\ge4$.
Hence $x_u-x_v-2-\eta>0$ in \eqref{eq:tour-arc}.
Under the second hypothesis, $|\eta|\le1$ and
$x_u-x_v-2=2(d_u-d_v-1)$. The claimed bounds now follow from
\eqref{eq:tour-arc} and $\det C>0$.
\end{proof}

\subsection{Proof of regularity}

By \Cref{lem:tour-minimum,cor:tour-arc-improvement}, every maximizer of
odd order $n\ge5$ satisfies the local condition
\begin{equation}\label{eq:tour-local}
 u\to v\quad\Longrightarrow\quad d_u\le d_v+1.
\end{equation}
This local condition restricts the degree sequence to three consecutive values, providing the path used in the final balancing step.

\begin{lemma}\label{lem:tour-degree-path}
Suppose a tournament $T$ of order $n$ satisfies \eqref{eq:tour-local}.
Then its maximum
and minimum outdegrees differ by at most two. If $T$ has order $2d+1$
and is nonregular, its outdegrees belong to $\{d-1,d,d+1\}$, and there
is a directed path $u\to w\to v$ with
\[
 d_u=d+1,\qquad d_w=d,\qquad d_v=d-1.
\]
\end{lemma}

\begin{proof}
Let $M,m$ be the maximum and minimum outdegrees. Suppose $M-m\ge3$,
and let
\[
 H=\{i:d_i\ge M-1\},\qquad F=\{i:d_i\le m+1\}.
\]
These sets are disjoint. By \eqref{eq:tour-local}, all outneighbors of a
vertex of degree $M$ lie in $H$, and all inneighbors of a vertex of
degree $m$ lie in $F$. Consequently,
\[
 M\le|H|-1,\qquad n-1-m\le|F|-1.
\]
It follows that
\[
 M-m\le|H|+|F|-n-1\le-1,
\]
a contradiction.

Now let $n=2d+1$ and suppose $T$ is nonregular. The average outdegree is
$d$, so $m<d<M$. Since the degrees are integers and $M-m\le2$, the
extremes are $d-1,d+1$.
Choose $u,v$ with $d_u=d+1$ and $d_v=d-1$.
Condition \eqref{eq:tour-local} forces $v\to u$.
With $N_T^+(z),N_T^-(z)$ denoting the outneighbors and inneighbors of
$z$, let
\[
 P=N_T^+(u)\cap N_T^-(v),\qquad
 Q=N_T^+(v)\cap N_T^-(u).
\]
The common outneighbors cancel in the degree difference, and the arc
$v\to u$ contributes $-1$. Hence
\[
 2=d_u-d_v=|P|-|Q|-1.
\]
Thus $|P|-|Q|=3$, and we may choose $w\in P$.
The path $u\to w\to v$ and \eqref{eq:tour-local} give
$d_w\ge d_u-1=d$ and $d_w\le d_v+1=d$, proving $d_w=d$.
\end{proof}

The following lemma provides a strict improvement by reversing the two
arcs of such a path.

\begin{lemma}[Two-arc balancing]\label{lem:tour-two-arcs}
Let $n=2d+1\ge5$, and let $T$ have outdegrees in $\{d-1,d,d+1\}$.
Suppose $u\to w\to v$ is a directed path with degrees $d+1,d,d-1$ in
that order. Let $T_2$ be obtained by reversing both arcs of this path,
and let $C_0$ be the principal submatrix of $\mathcal B(T)$ outside $u,w,v$.
Then
\[
 \Psi(T_2)-\Psi(T)
 \ge\frac{2(n-1)}{n-2}\det C_0>0.
\]
\end{lemma}

\begin{proof}
Order the vertices as $u,w,v$, followed by the remaining $n-3$ vertices.
Since $u\to w\to v$, the matrix in this order is
\[
 \mathcal B(T)=
 \begin{pmatrix}
 n+2&-1&\varepsilon&p^{\mathsf T}\\
 1&n&-1&r^{\mathsf T}\\
 -\varepsilon&1&n-2&q^{\mathsf T}\\
 -p&-r&-q&C_0
 \end{pmatrix},
\]
where $\varepsilon\in\{-1,1\}$ records the orientation between $u,v$,
and $p,r,q\in\{-1,1\}^{n-3}$ record their arcs to the other vertices.
First reverse $u\to w$, obtaining $T_1$, and then reverse $w\to v$.
The degree triples at $u,w,v$ change as
\[
 (d+1,d,d-1)\longmapsto(d,d+1,d-1)\longmapsto(d,d,d).
\]
The corresponding diagonal triples before the two reversals are
$(n+2,n,n-2)$ and $(n,n+2,n-2)$. Thus every diagonal entry is at least
$n-2$ at both steps, and the degree differences for the two reversals
are $1$ and $2$.

The first reversal changes the leading $2\times2$ block from
$\left(\begin{smallmatrix}n+2&-1\\1&n\end{smallmatrix}\right)$ to
$\left(\begin{smallmatrix}n&1\\-1&n+2\end{smallmatrix}\right)$ and leaves
all other entries unchanged. Let $V_0=V(T)\setminus\{u,w,v\}$. Let $C_1$
be the principal submatrix of $\mathcal B(T)$ on $\{v\}\cup V_0$, and let
$C_2$ be that of $\mathcal B(T_1)$ on $\{u\}\cup V_0$. These are the
matrices outside the endpoints of the first and second reversals.
Ordering $v$ and $u$ first, respectively, gives
\[
 C_1=\begin{pmatrix}n-2&q^{\mathsf T}\\-q&C_0\end{pmatrix},
 \qquad
 C_2=\begin{pmatrix}n&p^{\mathsf T}\\-p&C_0\end{pmatrix}.
\]
In particular, the diagonal entry at $u$ in $C_2$ is $n$,
its value after the first reversal.
By \Cref{cor:tour-arc-improvement},
\[
 \Psi(T_1)-\Psi(T)\ge-2\det C_1,
 \qquad
 \Psi(T_2)-\Psi(T_1)\ge2\det C_2.
\]
To compare the possible loss with the guaranteed gain, use their common
principal submatrix $C_0$. This matrix is unchanged by the reversals.
Its diagonal entries are at least $n-2>0$, so \Cref{lem:tour-cayley}
gives $\det C_0>0$. The Schur identity in \Cref{lem:sylvester}
applied to the two displayed matrices gives
\[
 \det C_1=\det C_0\bigl(n-2+q^{\mathsf T}C_0^{-1}q\bigr),
 \qquad
 \det C_2=\det C_0\bigl(n+p^{\mathsf T}C_0^{-1}p\bigr).
\]
Since $\|q\|^2=n-3$, the quadratic-form estimate in
\Cref{lem:tour-cayley}, with $\mu=n-2$, yields
\[
 p^{\mathsf T}C_0^{-1}p\ge0,
 \qquad
 q^{\mathsf T}C_0^{-1}q\le\frac{n-3}{n-2}.
\]
Therefore
\[
 \det C_2-\det C_1
 \ge\det C_0\left(2-\frac{n-3}{n-2}\right)
 =\frac{n-1}{n-2}\det C_0.
\]
Adding the two reversal bounds proves the asserted estimate.
The strict positivity follows from \Cref{lem:tour-cayley}.
\end{proof}

\begin{proof}[Proof of \Cref{thm:regular-main}]
Let $n=2d+1\ge5$, and choose a tournament $T$ maximizing $\tau(T)$.
By \Cref{lem:tour-minimum,cor:tour-arc-improvement}, $T$ satisfies
\eqref{eq:tour-local}. If $T$ were nonregular,
\Cref{lem:tour-degree-path} would place all its degrees in
$\{d-1,d,d+1\}$ and give a directed path with degrees $d+1,d,d-1$.
Reversing its two arcs strictly increases $\Psi$, by
\Cref{lem:tour-two-arcs}, and hence strictly increases $\tau(T)$.
This contradicts maximality.

For $n=1$, the unique tournament is regular and has one arborescence,
consisting of its single vertex. For $n=3$, the directed $3$-cycle has total count $3$,
whereas the transitive tournament has total count $2$. These two isomorphism types establish the assertion for $n=3$.
\end{proof}

﻿\section{Complete bipartite graphs}\label{sec:problem55}
Problem~5.5 of \cite{bandekar2025} asks for the maximum rooted count in
Eulerian orientations of $K_{n,m}$ and the maximum total count in arbitrary
orientations. For arbitrary orientations, a common determinant comparison
gives the maximum and all maximizers when $m\ge n(n-1)$, and a sharp
design bound when $nm/(n+m)$ is an integer. For Eulerian orientations, integer common-neighbor counts
strengthen this comparison, with equality in two infinite families.
Throughout this section, $2\le n\le m$.

\subsection{A symmetric determinant comparison}
Let the parts be $U=\{u_1,\ldots,u_n\}$ and $V=\{v_1,\ldots,v_m\}$.
Let $C\in\{0,1\}^{n\times m}$ encode the arcs from $U$ to $V$ by
\[
 C_{ij}=1\quad\Longleftrightarrow\quad u_i\longrightarrow v_j.
\]
Let
\[
 r_i=\sum_j C_{ij},\qquad k_j=\sum_i C_{ij},\qquad q_j=n-k_j.
\]
Thus $r_i=d_D^+(u_i)$ and $q_j=d_D^+(v_j)$. Let
$R=\diag(r_i)$, $Q=\diag(q_j)$ and $P=I_n-J_n/n$.
The next lemma replaces the nonsymmetric Laplacian by a positive
semidefinite matrix whose trace is determined by the column sums.

\begin{lemma}\label{lem:bipartite-determinant}
Every orientation $D$ of $K_{n,m}$ satisfies
\[
 n\tau(D)=\det\begin{pmatrix}R+J_n&C\\-C^{\T}&Q\end{pmatrix}.
\]
If $q_j>0$ for every $j$, let
\[
 W=R+CQ^{-1}C^{\T},\qquad K=\sum_j k_j,\qquad H=\sum_j\frac{k_j}{q_j}.
\]
Then
\begin{align}
 \tau(D)&=\frac{\prod_jq_j}{n}\det(W+J_n),\label{eq:bipartite-schur}\\
 \tr W&=K+H,\qquad \frac1n\one^{\T}W\one=H,
 \qquad \tr(PWP)=K,\label{eq:bipartite-traces}
\end{align}
and
\[
 \tau(D)\le\Phi(q):=\frac{\prod_jq_j}{n}
       \left(\frac{K}{n-1}\right)^{n-1}(n+H).
\]
For $\tau(D)>0$, equality holds if and only if
\begin{equation}\label{eq:bipartite-profile-equality}
 W=\frac{K}{n-1}P+\frac{H}{n}J_n.
\end{equation}
\end{lemma}
\begin{proof}
By \Cref{lem:matrix-tree}, the out-Laplacian and its adjugate are
\[
 L_D=\begin{pmatrix}R&-C\\C^{\T}-J_{m,n}&Q\end{pmatrix},
 \qquad \adj L_D=\one t^{\T},\qquad t=(a_v(D))_v.
\]
Let $T=\diag(I_n,-I_m)$,
$w=T\one$ and $u=(\one_n,0)^{\T}$. Direct block multiplication gives
\[
 TL_DT+wu^{\T}
 =\begin{pmatrix}R&C\\J_{m,n}-C^{\T}&Q\end{pmatrix}
  +\begin{pmatrix}J_n&0\\-J_{m,n}&0\end{pmatrix}
 =\begin{pmatrix}R+J_n&C\\-C^{\T}&Q\end{pmatrix}.
\]
Since $T^2=I$ and $L_D\one=0$, we have
$\adj(TL_DT)=T(\adj L_D)T=wt^{\T}T$ and $\det(TL_DT)=0$.
The rank-one identity in \Cref{lem:sylvester} gives
\[
 \det(TL_DT+wu^{\T})
 =u^{\T}\adj(TL_DT)w=(u^{\T}w)(t^{\T}Tw)=n\tau(D).
\]
Here $u^{\T}w=n$, $Tw=\one$ and $t^{\T}\one=\tau(D)$.

If every $q_j>0$, the Schur complement identity in
\Cref{lem:sylvester} gives
\[
 n\tau(D)=\det Q\,\det(R+J_n+CQ^{-1}C^{\T}),
\]
which proves \eqref{eq:bipartite-schur}. For the columns $s_j$ of $C$,
\[
 W=\sum_{j=1}^m\left(\diag(s_j)+\frac{s_js_j^{\T}}{q_j}\right).
\]
Each summand is positive semidefinite, with trace and quadratic form
\[
 k_j+\frac{k_j}{q_j},\qquad
 \one^{\T}\left(\diag(s_j)+\frac{s_js_j^{\T}}{q_j}\right)\one
 =k_j+\frac{k_j^2}{q_j}=\frac{nk_j}{q_j},
\]
since $s_j^{\T}s_j=s_j^{\T}\one=k_j$ and $k_j+q_j=n$.
Summing gives $\tr W=K+H$ and $\one^{\T}W\one=nH$; hence
\[
 \tr(PWP)=\tr(WP)=\tr W-\frac1n\one^{\T}W\one=K,
\]
by $P^2=P$ and cyclic invariance of trace. This proves
\eqref{eq:bipartite-traces}.

Let $\overline W$ be the average of $SWS^{\T}$ over all permutation
matrices $S$. Its diagonal and off-diagonal entries are constant, so
$\overline W=aP+bJ_n/n$. The permutation-invariant identities
\eqref{eq:bipartite-traces} give $(n-1)a=K$ and $b=H$, whence
\[
 \overline W=\frac K{n-1}P+\frac HnJ_n.
\]
If $\tau(D)>0$, the positive semidefinite matrix $W+J_n$ has positive
determinant by \eqref{eq:bipartite-schur}, and is therefore positive
definite. Strict concavity of $\log\det$ follows from its second
derivative along $A+tB$, for symmetric $B$:
\[
 -\tr\bigl((A+tB)^{-1}B(A+tB)^{-1}B\bigr)
 =-\bigl\|(A+tB)^{-1/2}B(A+tB)^{-1/2}\bigr\|_F^2.
\]
This is negative unless $B=0$. As $S J_n S^{\T}=J_n$, concavity gives
\[
 \log\det(W+J_n)\le\log\det(\overline W+J_n).
\]
The eigenvalues of $\overline W+J_n$ are $n+H$ on the span of $\one$
and $K/(n-1)$ on $\one^\perp$, giving the asserted bound by
\eqref{eq:bipartite-schur}. Strict concavity gives equality exactly
when all $SWS^{\T}$ coincide, or $W=\overline W$, as in
\eqref{eq:bipartite-profile-equality}. For $\tau(D)=0$, the bound
follows from its nonnegative right-hand side.
\end{proof}

A zero value of $q_j$ corresponds to a sink in $V$ and requires a
cofactor in place of the Schur complement above.
\begin{lemma}\label{lem:bipartite-sink}
If $D$ has two sinks, then $\tau(D)=0$. If $v_m$ is its unique sink,
delete column $m$ from $C$, and let $R'$, $Q'$ and $W'$ be the
corresponding matrices for the remaining matrix $C'$. Let $K'=\sum_{j<m}k_j$. Then
\[
 \tau(D)=\left(\prod_{j<m}q_j\right)
 \det\left((I_n+PW'P)|_{\one^\perp}\right)
 \le\left(\prod_{j<m}q_j\right)
       \left(1+\frac{K'}{n-1}\right)^{n-1}.
\]
\end{lemma}
\begin{proof}
Every nonroot vertex of an arborescence has an outgoing arc, so two
sinks preclude an arborescence and a unique sink must be its root.
If $v_m$ is the unique sink, its column is $\one_n$, all $q_j>0$
for $j<m$, and $R=R'+I_n$. Thus
\Cref{lem:matrix-tree} gives
\[
 \tau(D)=a_{v_m}(D)
 =\det\begin{pmatrix}R'+I_n&-C'\\(C')^{\T}-J_{m-1,n}&Q'\end{pmatrix}.
\]
The Schur complement identity in \Cref{lem:sylvester} gives
\[
 \tau(D)=\left(\prod_{j<m}q_j\right)
 \det\bigl(I_n+W'-C'(Q')^{-1}\one_{m-1}\one_n^{\T}\bigr).
\]
For each remaining column $s_j$, the identity $q_j+k_j=n$ gives
\[
 \left(\diag(s_j)+\frac{s_js_j^{\T}}{q_j}\right)\one_n
 =\left(1+\frac{k_j}{q_j}\right)s_j=\frac n{q_j}s_j.
\]
Summing over $j<m$ yields
$W'\one_n=nC'(Q')^{-1}\one_{m-1}$. The matrix inside the
determinant is therefore $I_n+W'-W'J_n/n=I_n+W'P$.

In an orthonormal basis beginning with $\one_n/\sqrt n$, the first
column of $I_n+W'P$ is $(1,0,\ldots,0)^{\T}$, since $P\one_n=0$.
The remaining diagonal block is $(I_n+PW'P)|_{\one^\perp}$;
expansion along the first column proves the determinant identity.
By the column calculation in \Cref{lem:bipartite-determinant}, the
compression of $W'$ has nonnegative eigenvalues $\mu_1,\ldots,\mu_{n-1}$
summing to $K'$. The arithmetic--geometric mean inequality gives
\[
 \det\left((I_n+PW'P)|_{\one^\perp}\right)
 =\prod_{i=1}^{n-1}(1+\mu_i)
 \le\left(\frac{\sum_{i=1}^{n-1}(1+\mu_i)}{n-1}\right)^{n-1}
 =\left(1+\frac{K'}{n-1}\right)^{n-1}.
\]
\end{proof}

\subsection{Exact maximizers for \texorpdfstring{$m\ge n(n-1)$}{m >= n(n-1)}}
A fixed collection of singleton columns attains the determinant bound
once the larger part has enough vertices to contain it.
\begin{theorem}\label{thm:bipartite-large}
Every orientation $D$ of $K_{n,m}$ satisfies
\[
 \tau(D)\le 2n^{n+m-1}\left(\frac{n-1}{n}\right)^{n(n-1)}.
\]
Equality is possible if and only if $m\ge n(n-1)$. In this range,
\[
 \max_D\tau(D)=2n^{n-1}(n-1)^{n(n-1)}n^{m-n(n-1)},
\]
and every maximizing matrix $C$ has $n-1$ copies of each standard
basis vector $e_i$, together with $m-n(n-1)$ zero columns.
Consequently, the maximizing orientation is unique up to permutations
within the two parts.
\end{theorem}
\begin{proof}
We may assume $\tau(D)>0$, since the bound is positive. Let $p=(n-1)/n$.

Assume first that every $q_j>0$. Then $K>0$, since $K=0$ makes
all $n\ge2$ vertices in $U$ sinks, contrary to \Cref{lem:bipartite-sink}.
By \Cref{lem:bipartite-determinant},
\[
 \frac{\tau(D)}{n^m}\le
 \frac1n\left(\frac K{n-1}\right)^{n-1}(n+H)
 \prod_j\left(1-\frac{k_j}{n}\right).
\]
To obtain an upper bound depending only on $K$, replace an entry $k$
with $2\le k<n$ in the list of column sizes by $k$ entries equal to $1$.
This preserves $K$ but changes the number of entries. Zero entries
contribute $1$ to the product and $0$ to $H$, so they do not affect
the right-hand side.

If $h$ is the contribution of the other columns to $H$, the ratio
of the new expression to the old one is
\[
 \frac{p^k}{1-k/n}\,
 \frac{n+h+k/(n-1)}{n+h+k/(n-k)}.
\]
Since $k/(n-1)<k/(n-k)$, the ratio increases with $h\ge0$.
It therefore suffices to prove the inequality at $h=0$:
\[
 p^k\left(n+\frac{k}{n-1}\right)
 >\left(1-\frac kn\right)\left(n+\frac{k}{n-k}\right).
\]
Let $x=1/n$. After division by $n$, the difference of the two sides is
$g_k=p^{k-1}(p+kx^2)-1+kxp$. Since $p=1-x$, we have
$g_1=p+x^2-1+xp=0$, and expansion gives
\[
 g_{k+1}=pg_k+x^2(p^k-1+kp).
\]
For $n\ge3$, $p\ge2/3$: the parenthesis is $2p-1>0$ at $k=1$,
and is positive for $k\ge2$ because $kp>1$. The recurrence gives
$g_k>0$ for $k\ge2$, so every replacement strictly increases the
bound. When $n=2$, all column sizes with $q_j>0$ are already zero or one.

Consequently,
\[
 \frac{\tau(D)}{n^m}\le F_n(K):=\frac{p^K}{n}\left(\frac K{n-1}\right)^{n-1}
              \left(n+\frac K{n-1}\right).
\]
We next maximize the upper bound $F_n(K)$ over the positive integers.
For real $K>0$,
\begin{align*}
 (\log F_n)'(K)&=\log p+\frac{n-1}{K}
                         +\frac1{n(n-1)+K},\\
 (\log F_n)''(K)&=-\frac{n-1}{K^2}
                         -\frac1{(n(n-1)+K)^2}<0.
\end{align*}
Thus $\log F_n$ is strictly concave. At $K_0=n(n-1)$,
\[
 (\log F_n)'(K_0)=\log(1-1/n)+\frac1n+\frac1{2n(n-1)}
 =\sum_{j\ge3}\left(\frac12-\frac1j\right)n^{-j}>0.
\]
using $\log(1-x)=-\sum_{j\ge1}x^j/j$ and
$x^2/(2(1-x))=\sum_{j\ge2}x^j/2$ at $x=1/n$.
Since $(\log F_n)'$ decreases and is positive at $K_0$, $F_n$
increases up to $K_0$. For the next integer, let $b=n-1$; then
\[
 \frac{F_n(K_0+1)}{F_n(K_0)}
 =p\left(1+\frac1{nb}\right)^b\left(1+\frac1{2nb}\right)<1.
\]
To verify the strict inequality, for $z>0$ use
$\log(1+z)<z-z^2/2+z^3/3$ and $\log(1+z)<z$.
For $f_1(z)=z-z^2/2+z^3/3-\log(1+z)$ and
$f_2(z)=z-\log(1+z)$, we have $f_1(0)=f_2(0)=0$ and
$f_1'(z)=z^3/(1+z)>0$, $f_2'(z)=z/(1+z)>0$ for $z>0$.
Applying these bounds at $1/(nb)$ and $1/(2nb)$ bounds the logarithm
of the ratio strictly from above by
\begin{align*}
 &\log p+\frac1n-\frac1{2n^2b}+\frac1{3n^3b^2}+\frac1{2nb}\\
 &=\log p+\frac1n+\frac1{2n^2}+\frac1{3n^3b^2}
 \le\log p+\frac1n+\frac1{2n^2}+\frac1{3n^3}<0.
\end{align*}
Here $b\ge1$, and the last inequality follows from the negative
terms remaining in the series for $\log(1-1/n)$ after the first
three cancel. Strict concavity makes the integer increments of
$\log F_n$ decrease, so all increments after $K_0$ are negative.
Thus $K_0$ is the unique positive integer maximizing $F_n$, and
$\tau(D)/n^m\le F_n(K_0)=2n^{n-1}p^{n(n-1)}$.

It remains to consider an orientation with a sink in $V$.
The sink is unique because $\tau(D)>0$. Bernoulli's inequality
$p^k\ge1-k/n$, for integers $0\le k<n$, applied to the $m-1$
remaining columns in \Cref{lem:bipartite-sink} gives
\[
 \frac{\tau(D)}{n^m}\le G_n(K'):=\frac{p^{K'}}n
                       \left(1+\frac{K'}{n-1}\right)^{n-1}.
\]
Direct cancellation gives
\[
 \frac{F_n(K'+n-1)}{G_n(K')}
 =\left(n+1+\frac{K'}{n-1}\right)p^{n-1}>1,
\]
since Bernoulli's inequality gives $p^{n-1}\ge1-(n-1)/n=1/n$
and hence a ratio at least $(n+1+K'/(n-1))/n>1$.
As $K'+n-1$ is a positive integer,
\[
 G_n(K')<F_n(K'+n-1)\le F_n(K_0).
\]
The bound is therefore strict in this case.

The preceding comparisons also determine all equality cases.
Equality requires singleton or zero columns, $K=K_0$, and
\eqref{eq:bipartite-profile-equality}, by the column-size comparison,
the maximization of $F_n$ over the positive integers, and the determinant bound.
Let $h_i$ be the number of copies of $e_i$. A singleton contributes
$ne_ie_i^{\T}/(n-1)$ to $W$, so
$W=n\diag(h_i)/(n-1)$ and $H=K_0/(n-1)=n$.
Thus \eqref{eq:bipartite-profile-equality} gives $W=nP+J_n=nI_n$,
so $h_i=n-1$ for each $i$. These $K_0$ columns, followed by zero
columns, exist exactly when $m\ge K_0$.

Conversely, the prescribed columns give $W=nI_n$,
$\prod_jq_j=(n-1)^{K_0}n^{m-K_0}$ and
$\det(nI_n+J_n)=2n^n$ from its eigenvalues $2n$ and $n$
(the latter with multiplicity $n-1$). Equation~\eqref{eq:bipartite-schur} gives
\[
 \tau(D)=2n^{n-1}(n-1)^{K_0}n^{m-K_0}
       =2n^{n+m-1}p^{K_0}.
\]
This unique maximizing column multiset gives uniqueness up to
permutations within the two parts.
\end{proof}

\subsection{Sharp bounds and incidence designs}
When the average outdegree $d=nm/(n+m)$ is integral, an upper bound
obtained from the concavity of $\log\det$ has equality conditions
characterized by incidence designs.
A design with parameters $(v,b,k,r,\lambda)$ consists of $b$ indexed
$k$-subsets of a $v$-point set, each point occurring in $r$ blocks
and each pair in $\lambda$ blocks. Repeated blocks correspond to
distinct vertices of $K_{n,m}$; we allow $k=1$ and $\lambda=0$.

\begin{theorem}\label{thm:bipartite-integral}
Suppose $d=nm/(n+m)$ is an integer. Every orientation $D$ of $K_{n,m}$
satisfies
\[
 \tau(D)\le 2d^{n+m-1}\left(\frac n{n-1}\right)^{n-1}.
\]
Equality holds if and only if $C$ is the incidence matrix of a design
with parameters
\begin{equation}\label{eq:bipartite-design-parameters}
 (v,b,k,r,\lambda)=\left(n,m,n-d,d,\frac{d(n-d-1)}{n-1}\right).
\end{equation}
Equivalently,
\[
 C\one=d\one,\qquad C^{\T}\one=(n-d)\one,\qquad
 CC^{\T}=\frac{d^2}{n-1}I_n+\frac{d(n-d-1)}{n-1}J_n.
\]
\end{theorem}
\begin{proof}
Let $h=n+m$ and $c_n=(n/(n-1))^{n-1}$. Since $m\ge n$ and
$d=nm/(n+m)$ is an integer,
\[
 \frac n2\le d\le n-1,\qquad
 m=\frac{nd}{n-d}\le n(n-1),\qquad m(n-d)=nd.
\]
We may assume $\tau(D)>0$.

If $D$ has a sink, \Cref{lem:bipartite-sink} makes it the unique root.
Each arborescence chooses one outgoing arc at every other vertex;
these $h-1$ outdegrees sum to $nm=hd$. The arithmetic--geometric
mean inequality gives
\[
 \tau(D)\le\prod_{v\ne\mathrm{sink}}d_D^+(v)
 \le\left(\frac{nm}{h-1}\right)^{h-1}
 =d^{h-1}\left(1+\frac1{h-1}\right)^{h-1}
 <e d^{h-1}.
\]
The last inequality follows from $\log(1+x)<x$.
Bernoulli's inequality gives $c_n=(1+1/(n-1))^{n-1}\ge2$, so $e<2c_n$ and the
asserted bound is strict.

Otherwise all outdegrees are positive, so
\Cref{lem:bipartite-determinant} applies. Let
\[
 Z=W+J_n,\qquad \alpha=\frac{nd}{n-1},\qquad
 Z_0=\alpha P+2J_n.
\]
By \eqref{eq:bipartite-schur}, the positive semidefinite matrix $Z$
has positive determinant and is therefore positive definite.
The eigenvalues of $Z_0$ are $\alpha$ on $\one^\perp$ and $2n$
on the span of $\one$, giving
\[
 Z_0^{-1}=\frac{n-1}{nd}P+\frac{J_n}{2n^2},\qquad
 \det Z_0=2n\left(\frac{nd}{n-1}\right)^{n-1},\qquad
 \frac{d^m}{n}\det Z_0=2c_nd^{h-1}.
\]
Let $\rho_1,\ldots,\rho_n$ be the positive eigenvalues of
$Z_0^{-1/2}ZZ_0^{-1/2}$. The inequality $\log x\le x-1$ gives
\[
 \log\det Z-\log\det Z_0
 =\sum_{i=1}^n\log\rho_i
 \le\sum_{i=1}^n(\rho_i-1)=\tr(Z_0^{-1}Z)-n.
\]
Equality means every $\rho_i=1$, equivalently $Z=Z_0$ by symmetry.
Using $PJ_n=0$ and \eqref{eq:bipartite-traces},
\begin{align*}
 \tr(Z_0^{-1}Z)
 &=\frac{n-1}{nd}\tr(PW)
   +\frac1{2n^2}\tr(J_nW+J_n^2)\\
 &=\frac{n-1}{nd}K+\frac{H}{2n}+\frac12.
\end{align*}
At $q_j=d$, we have $K=m(n-d)=nd$ and $H=m(n-d)/d=n$, so
\[
 \tr(Z_0^{-1}Z)-n
 =\frac{n-1}{nd}(K-nd)+\frac{H-n}{2n}.
\]
The definitions of $K,H$ and $m(n-d)=nd$ also give
\[
 K-nd=\sum_{j=1}^m(d-q_j),\qquad
 H-n=n\sum_{j=1}^m\left(\frac1{q_j}-\frac1d\right).
\]
Substitution and \eqref{eq:bipartite-schur} yield
\begin{align*}
 \log\frac{\tau(D)}{2c_nd^{h-1}}
 &=\sum_{j=1}^m\log\frac{q_j}{d}
      +\log\det Z-\log\det Z_0\\
 &\le\sum_{j=1}^m\left[
   \log\frac{q_j}{d}-\frac{n-1}{nd}(q_j-d)
             +\frac12\left(\frac1{q_j}-\frac1d\right)\right]\\
 &=\sum_{j=1}^m\bigl(f(q_j)-f(d)\bigr),
 \qquad f(x)=\log x-\frac{n-1}{nd}x+\frac1{2x}.
\end{align*}

To maximize $f$ on $\{1,\ldots,n\}$, observe that
\[
 f''(x)=\frac{1-x}{x^3}\le0\quad(x\ge1),\qquad
 f'(d)=\frac{2d-n}{2nd^2}\ge0.
\]
For $d>1$, strict decrease of $f'$ gives $f'(x)>f'(d)\ge0$
when $1\le x<d$, hence $f(q)<f(d)$ for integers $q<d$;
there are none when $d=1$. On the other side, use
\[
 \log(1+x)<x-\frac{x^2}{2(1+x)}\qquad(x>0).
\]
Indeed, put $g(x)=x-x^2/(2(1+x))-\log(1+x)$. Then $g(0)=0$ and
$g'(x)=x^2/(2(1+x)^2)>0$ for $x>0$. At $x=1/d$, this gives
\begin{align*}
 f(d+1)-f(d)
 &=\log(1+1/d)-\frac1d+\frac1{nd}-\frac1{2d(d+1)}\\
 &<\frac1{nd}-\frac1{d(d+1)}\le0,
\end{align*}
since $d+1\le n$. Concavity makes subsequent integer increments
nonincreasing. Thus $d$ is the unique maximizer of $f$ over $\{1,\ldots,n\}$, and
$f(q_j)-f(d)\le0$ proves the bound.

Equality requires no sink, $q_j=d$ for every $j$, and $Z=Z_0$.
Hence $C^{\T}\one=(n-d)\one$, $Q=dI_m$, and comparison of
diagonal entries gives
\[
 \left(1+\frac1d\right)r_i+1=d+2,
 \qquad r_i=d.
\]
For $i\ne j$, let $c_{ij}=(CC^{\T})_{ij}$. Off the diagonal,
\[
 1+\frac{c_{ij}}d=2-\frac d{n-1},\qquad
 c_{ij}=\frac{d(n-d-1)}{n-1}.
\]
These are the block size, replication and concurrence in
\eqref{eq:bipartite-design-parameters}. Since
$d-d(n-d-1)/(n-1)=d^2/(n-1)$, they also give the stated formula
for $CC^{\T}$.

Conversely, those incidence identities give $R=dI_n$, $Q=dI_m$ and
\[
 W+J_n=dI_n+\frac1dCC^{\T}+J_n
       =\frac{nd}{n-1}P+2J_n=Z_0.
\]
Equation~\eqref{eq:bipartite-schur} now gives
$\tau(D)=d^m\det Z_0/n=2c_nd^{h-1}$, proving the converse.
\end{proof}

Whenever the design exists, \Cref{thm:bipartite-integral} determines
the maximum and all maximizers. At $d=n-1$ it recovers
\Cref{thm:bipartite-large}; the next family attains the bound
strictly below that boundary.
\begin{corollary}\label{cor:bipartite-mcfarland}
For every prime power $s\ge2$, let
\[
 L=s^2+s-1,\qquad n=sL,\qquad m=(s^2-1)L,\qquad d=s(s^2-1).
\]
Then $n<m<n(n-1)$, and equality holds in the upper bound of
\Cref{thm:bipartite-integral}. The maximizing orientations are exactly
those whose matrices $C$ are incidence matrices of designs with
parameters $(n,m,s^2,d,s(s-1))$.
\end{corollary}
\begin{proof}
We use McFarland's two-dimensional difference-set construction
\cite{mcfarland1973}. Let $V_0=\FF_s^2$ and $Z=\ZZ/(s+2)\ZZ$.
Index the $(s^2-1)/(s-1)=s+1$ one-dimensional subspaces of $V_0$
by the nonzero elements $i$ of $Z$, denoting them by $H_i$. Let
\[
 G=V_0\times Z,\qquad
 \mathcal D=\bigcup_{i\ne0}(H_i\times\{i\}).
\]
Disjointness in the second coordinate gives
$|G|=s^2(s+2)$ and $|\mathcal D|=s(s+1)$.

Every nonzero element of $G$ has $s$ representations as an ordered
difference of elements of $\mathcal D$. For $(v,0)$ with $v\ne0$,
the second coordinates agree, and the unique subspace containing $v$
contributes the $s$ pairs $(y+v,y)$ with $y\in H_i$; the others
contribute none. For $(v,z)$ with $z\ne0$, the labels satisfy
$i-j=z$ and $i,j\ne0$. Excluding $j=0,-z$ leaves $s$ pairs
$(i,j)$ with $i\ne j$. Since $H_i\cap H_j=\{0\}$ and
$H_i+H_j=V_0$, each pair gives a unique solution of $x-y=v$
with $x\in H_i$, $y\in H_j$, again yielding $s$ representations.

The blocks $\mathcal D+g$, indexed by $g\in G$, have size $s(s+1)$.
Each point $x$ lies in $s(s+1)$ blocks, one for each $g=x-d$ with
$d\in\mathcal D$. For distinct points $x,y$, the blocks containing
both correspond to pairs $d_1-d_2=x-y$ in $\mathcal D$, so there
are $s$. The same difference count gives intersection size $s$ for
blocks with different labels. Thus the design has parameters
\[
 (v,b,k,r,\lambda)
 =\bigl(s^2(s+2),s^2(s+2),s(s+1),s(s+1),s\bigr).
\]

Delete the block $\mathcal D$ and restrict the others to
$G\setminus\mathcal D$. Each loses $s$ points, leaving size
$s(s+1)-s=s^2$. Points and pairs outside $\mathcal D$ retain
replication $s(s+1)$ and concurrence $s$, since they did not occur
in the deleted block. With $L=s^2+s-1$,
\[
 v'=s^2(s+2)-s(s+1)=sL,\qquad b'=s^2(s+2)-1=(s+1)L.
\]
The residual parameters are $(sL,(s+1)L,s^2,s(s+1),s)$.
Repeating every block $s-1$ times multiplies $b',r,\lambda$ by
$s-1$, giving
\[
 (n,m,k,r,\lambda)
 =\bigl(sL,(s^2-1)L,s^2,s(s^2-1),s(s-1)\bigr).
\]

Since $s+(s^2-1)=L$,
\[
 n+m=L^2,\qquad nm=s(s^2-1)L^2,\qquad
 \frac{nm}{n+m}=s(s^2-1)=d.
\]
Also
\[
 n-d=s^2,\qquad
 n-1=(s-1)(s+1)^2,
 \qquad
 \frac{d(n-d-1)}{n-1}
 =\frac{s(s^2-1)^2}{(s-1)(s+1)^2}=s(s-1).
\]
These are the parameters in \eqref{eq:bipartite-design-parameters},
so \Cref{thm:bipartite-integral} gives attainment and all maximizers.
Finally, $m/n=(s^2-1)/s>1$ for $s\ge2$. Since $d=n-s^2<n-1$
and $m=nd/(n-d)$,
\[
 m<n(n-1)\quad\Longleftrightarrow\quad
 d<(n-1)(n-d)\quad\Longleftrightarrow\quad d<n-1.
\]
\end{proof}

\subsection{Spectral bounds for Eulerian orientations}
Assume that $n$ and $m$ are even. For an Eulerian orientation, let
\[
 X=2C-J_{n,m},\qquad G=XX^{\T},\qquad
 A_{n,m}=\frac{m^{n-1}n^{m-1}}{2^{n+m-1}}.
\]
Eulerian balance gives row sums $m/2$ and column sums $n/2$ in $C$,
so $X\one_m=0$ and $X^{\T}\one_n=0$. Thus $G\one=0$,
each diagonal entry of $G$ is the sum of $m$ squared signs, namely
$m$, and $\tr G=nm$. The determinant identity specializes to
\begin{equation}\label{eq:eulerian-bipartite-determinant}
 a(D)=A_{n,m}\det\left(I_n+\frac{G}{nm}\right).
\end{equation}
Indeed, $R=(m/2)I_n$, $Q=(n/2)I_m$, and expansion of $XX^{\T}$
using $CJ_{m,n}=(m/2)J_n$ and $J_{n,m}J_{m,n}=mJ_n$ gives
$G=4CC^{\T}-mJ_n$. Substitution in \Cref{lem:bipartite-determinant} yields
\[
 W+J_n=\frac m2I_n+\frac2nCC^{\T}+J_n
       =\frac m2\left(I_n+\frac{G}{nm}\right)
          +\left(1+\frac m{2n}\right)J_n.
\]
On the span of $\one$ this has eigenvalue $m/2+n+m/2=n+m$;
for $x\in\one^\perp$, we have $J_nx=\one(\one^{\T}x)=0$.
Since $I_n+G/(nm)$ has
eigenvalue one on the span of $\one$,
\[
 \det(W+J_n)=(n+m)\left(\frac m2\right)^{n-1}
                    \det\left(I_n+\frac G{nm}\right).
\]
Together with $\prod_jq_j=(n/2)^m$, \eqref{eq:bipartite-schur} and
$\tau(D)=(n+m)a(D)$ from \Cref{lem:root-independence}, this gives
\[
 a(D)=\frac1n\left(\frac n2\right)^m\left(\frac m2\right)^{n-1}
               \det\left(I_n+\frac G{nm}\right),
\]
with coefficient $A_{n,m}$, proving \eqref{eq:eulerian-bipartite-determinant}.
For $n=2$, every column of $X$ has opposite signs, so
$G=m\left(\begin{smallmatrix}1&-1\\-1&1\end{smallmatrix}\right)$,
with eigenvalues zero and $2m$. The determinant in
\eqref{eq:eulerian-bipartite-determinant} is therefore two;
as $A_{2,m}=m/4$, every Eulerian orientation has $a(D)=m/2$.

For $n\ge4$, the next lemma bounds the product over the $n-1$
eigenvalues on $\one^\perp$ using their sum and variance.
\begin{lemma}\label{lem:eulerian-moment-product}
For $p\ge2$ and $0\le \nu\le(p-1)/p$, let
\[
 \begin{gathered}
 \alpha_p(\nu)=\frac1p-\sqrt{\frac{\nu}{p(p-1)}},\qquad
 \beta_p(\nu)=\frac1p+\sqrt{\frac{(p-1)\nu}{p}},\\
 \mathcal E_p(\nu)=(1+\alpha_p(\nu))^{p-1}(1+\beta_p(\nu)).
 \end{gathered}
\]
If $z_1,\ldots,z_p\ge0$, $\sum_i z_i=1$ and
$\nu=\sum_i(z_i-1/p)^2$, then
\[
 \prod_{i=1}^p(1+z_i)\le\mathcal E_p(\nu).
\]
For $\nu>0$, equality holds if and only if the multiset of the $z_i$
consists of one copy of $\beta_p(\nu)$ and $p-1$ copies of $\alpha_p(\nu)$.
For $\nu=0$, all $z_i=1/p$. The function $\mathcal E_p$ is strictly
decreasing on $[0,(p-1)/p]$.
\end{lemma}
To maintain the flow of the main text, the proof is given in \Cref{app:bipartite-moment-product}.

Integer common-neighbor counts give a lower bound on the spectral
variance, strengthening the product bound as follows.
\begin{theorem}\label{thm:eulerian-bipartite-bound}
Let $D$ be an Eulerian orientation of $K_{n,m}$, where $n,m$ are even
and $4\le n\le m$. Let
\[
 \lambda=\frac{m(n-2)}{4(n-1)},\qquad
 t=\lambda-\lfloor\lambda\rfloor,\qquad
 \nu_0=\frac{16(n-1)t(1-t)}{nm^2}.
\]
Then
\begin{equation}\label{eq:eulerian-moment-bound}
 \begin{aligned}
 a(D)&\le B_E(n,m):=A_{n,m}\mathcal E_{n-1}(\nu_0)\\
 &\le U_E(n,m):=A_{n,m}\left(\frac n{n-1}\right)^{n-1}.
 \end{aligned}
\end{equation}
The second inequality is strict if and only if $t>0$.
For $t=0$, equality $a(D)=B_E(n,m)$ holds if and only if
\[
 G=\frac{nm}{n-1}P.
\]
For $t>0$, equality holds if and only if there is a vector
$v\in\{1,-1\}^n$ with $v^{\T}\one=0$ such that
\begin{equation}\label{eq:eulerian-bipartite-two-level}
 G=\frac{n(m-2)}{n-1}P+2vv^{\T}.
\end{equation}
In this case $m\equiv2\pmod{2(n-1)}$ and
$t=(n-2)/(2(n-1))$. The two sign classes of $v$ have size $n/2$;
pairs within a class have $\lfloor\lambda\rfloor+1$ common
outneighbors in the larger part, and pairs in different classes have
$\lfloor\lambda\rfloor$.
\end{theorem}
\begin{proof}
Let $p=n-1$, $G_0=nmP/(n-1)$, and let $z_1,\ldots,z_p$ be the
eigenvalues of $G$ on $\one^\perp$, divided by $nm$. The matrix
$G=XX^{\T}$ is positive semidefinite, $G\one=0$, and
$\tr G=nm$. Thus $z_i\ge0$ and $\sum_i z_i=1$.

We first bound the variance using the integrality of the common-outneighbor counts.
Since $G_0=nmI/p$ on $\one^\perp$ and $G\one=G_0\one=0$,
the spectral theorem gives
\[
 \nu:=\sum_{i=1}^p(z_i-1/p)^2
 =\frac1{n^2m^2}\sum_{i=1}^p(nmz_i-nm/p)^2
 =\frac{\|G-G_0\|_F^2}{n^2m^2}.
\]
For $i\ne j$, let $c_{ij}=(CC^{\T})_{ij}$ count the common
outneighbors of $u_i,u_j$. The row and column sums of $C$ give
\[
 CC^{\T}\one=C\left(\frac n2\one\right)
                =\frac{nm}{4}\one,
 \qquad (CC^{\T})_{ii}=\frac m2.
\]
Subtracting the diagonal entry from the row sum gives
\[
 \sum_{j\ne i}c_{ij}=\frac{nm}{4}-\frac m2
                  =\frac{m(n-2)}4=(n-1)\lambda.
\]
As $G=4CC^{\T}-mJ_n$, $G$ and $G_0$ have diagonal entries $m$, and
\[
 (G-G_0)_{ij}=4c_{ij}-m+\frac m{n-1}
            =4(c_{ij}-\lambda),\qquad (G-G_0)_{ii}=0.
\]

Let $b=\lfloor\lambda\rfloor$, so $\lambda=b+t$. For integer $c$,
$(c-b)(c-b-1)\ge0$, with equality exactly at $b,b+1$.
Since $\sum_{j\ne i}(c_{ij}-\lambda)=0$, expansion gives
\begin{align*}
 0&\le\sum_{j\ne i}(c_{ij}-b)(c_{ij}-b-1)\\
  &=\sum_{j\ne i}(c_{ij}-\lambda)^2
    +(2t-1)\sum_{j\ne i}(c_{ij}-\lambda)
    +(n-1)t(t-1)\\
  &=\sum_{j\ne i}(c_{ij}-\lambda)^2-(n-1)t(1-t).
\end{align*}
Summing over all $n$ rows, with the Frobenius norm counting ordered pairs,
\[
 \|G-G_0\|_F^2
 =16\sum_{i=1}^n\sum_{j\ne i}(c_{ij}-\lambda)^2
 \ge16n(n-1)t(1-t).
\]
Division by $n^2m^2$ gives $\nu\ge\nu_0$, with equality exactly
when every off-diagonal $c_{ij}$ is $b$ or $b+1$.

To apply the product bound, note that $t(1-t)\le1/4$ and
$m\ge n\ge4$ give
\[
 0\le\nu_0\le\frac{4(n-1)}{nm^2}
 <\frac4{m^2}\le\frac14<\frac{p-1}{p}.
\]
Thus $\nu_0$ lies in the domain of $\mathcal E_p$.
By \eqref{eq:eulerian-bipartite-determinant},
\Cref{lem:eulerian-moment-product} and its strict monotonicity,
\[
 \frac{a(D)}{A_{n,m}}=\prod_{i=1}^p(1+z_i)
 \le\mathcal E_p(\nu)\le\mathcal E_p(\nu_0)
 \le\mathcal E_p(0)=\left(\frac n{n-1}\right)^{n-1}.
\]
This proves \eqref{eq:eulerian-moment-bound}. Since $0\le t<1$,
$\nu_0=0$ exactly when $t=0$; strict monotonicity gives $B_E<U_E$
exactly when $t>0$.

For the equality case with $t=0$, we have $\nu_0=0$, so equality
forces $\mathcal E_p(\nu)=\mathcal E_p(0)$,
hence $\nu=0$ and $G=G_0$. Conversely, $G=G_0$ gives $z_i=1/p$
and attains the bound. Equivalently, $CC^{\T}$ has diagonal $m/2$
and off-diagonal $\lambda$; with the column sums $n/2$, this is
the design with parameters $(n,m,n/2,m/2,\lambda)$.

Now suppose $t>0$ and the bound is attained. Equality in the
variance and product bounds gives $\nu=\nu_0$ and
$c_{ij}\in\{b,b+1\}$. Let $\alpha=\alpha_p(\nu_0)$ and
$\beta=\beta_p(\nu_0)>\alpha$. Equality in
\Cref{lem:eulerian-moment-product} gives one normalized eigenvalue
$\beta$ and $p-1$ copies of $\alpha$. For a unit eigenvector
$u\perp\one$ corresponding to the eigenvalue $nm\beta$, the spectral theorem gives
\[
 G=nm\alpha P+nm(\beta-\alpha)uu^{\T}.
\]
The diagonal entries give $1=p\alpha+n(\beta-\alpha)u_i^2$,
since $G_{ii}=m$ and $P_{ii}=p/n$. Together with
$(p-1)\alpha+\beta=1$, this yields
\[
 u_i^2=\frac{1-p\alpha}{n(\beta-\alpha)}=\frac1n.
\]
Thus $v=\sqrt n\,u$ is a sign vector with $n/2$ of each sign,
because $v^{\T}\one=0$. For $i\ne j$,
$G_{ij}=-m\alpha+m(\beta-\alpha)v_iv_j$ and
$c_{ij}=(G_{ij}+m)/4$, giving the following numbers of common
outneighbors for pairs with equal and opposite signs, respectively:
\[
 \frac m4(1-\alpha+\beta-\alpha),\qquad
 \frac m4(1-\alpha-(\beta-\alpha)).
\]
Their difference is $m(\beta-\alpha)/2>0$. Both types occur since
$n/2\ge2$, and both values belong to $\{b,b+1\}$; they must
therefore be $b+1$ and $b$. Hence $\beta-\alpha=2/m$.
Combining this with $p\alpha+(\beta-\alpha)=1$ yields
\[
 \alpha=\frac{m-2}{m(n-1)},\qquad
 \beta=\frac{m+2n-4}{m(n-1)},
\]
and substitution in the formula for $G$ proves
\eqref{eq:eulerian-bipartite-two-level}.
Each row has $n/2-1$ equal-sign pairs and $n/2$ opposite-sign pairs, so
\[
 \lambda=\frac{(n/2-1)(b+1)+(n/2)b}{n-1}
        =b+\frac{n-2}{2(n-1)}.
\]
This gives the stated fractional part. Letting $n=2k$ and $m=2q$ gives
\[
 b=\lambda-t=\frac{q(k-1)}{2k-1}-\frac{k-1}{2k-1}
             =\frac{(q-1)(k-1)}{2k-1}\in\ZZ.
\]
Since $(2k-1)-2(k-1)=1$, the factors $k-1$ and $2k-1$ are
coprime. Hence $2k-1$ divides $q-1$, or $m\equiv2\pmod{2(n-1)}$.

Conversely, suppose \eqref{eq:eulerian-bipartite-two-level} holds
for a balanced sign vector $v$. For distinct $i,j$, it gives
\[
 c_{ij}=\frac{G_{ij}+m}{4}
 =\begin{cases}
 \displaystyle\frac{(m-2)(n-2)}{4(n-1)}+1,&v_i=v_j,\\[4pt]
 \displaystyle\frac{(m-2)(n-2)}{4(n-1)},&v_i=-v_j.
 \end{cases}
\]
The lower value is an integer. With $n=2k,m=2q$, it is
$(q-1)(k-1)/(2k-1)$, so the same coprimality argument gives
$m\equiv2\pmod{2(n-1)}$. The row mean exceeds the lower value
by $(n-2)/(2(n-1))\in(0,1)$; hence the lower value is
$b=\lfloor\lambda\rfloor$. For each row, the common-outneighbor counts
attain the lower bound for the sum of their squared deviations from $\lambda$,
giving $\nu=\nu_0$.

Finally, $Pv=v$ and $v^{\T}v=n$, so
\eqref{eq:eulerian-bipartite-two-level} has eigenvalue
$n(m+2n-4)/(n-1)$ on $v$ and $n(m-2)/(n-1)$ with multiplicity
$p-1=n-2$ on its orthogonal complement in $\one^\perp$.
Dividing by $nm$ gives one upper and $p-1$ equal lower values with
sum $1$ and variance $\nu_0$, namely $\beta_p(\nu_0)$ and
$\alpha_p(\nu_0)$. Thus \Cref{lem:eulerian-moment-product} and
the variance comparison are both equalities, proving $a(D)=B_E(n,m)$.
\end{proof}

The next two families attain $B_E$ with $t>0$, giving exact maxima
where \eqref{eq:eulerian-moment-bound} strictly improves $U_E$.
\begin{samepage}
\begin{corollary}\label{cor:eulerian-sharp-families}
Let $h\ge1$ be an integer. Equality in the first inequality of
\eqref{eq:eulerian-moment-bound} is attained in each of the following
families:
\begin{enumerate}[label=\textup{(\roman*)}]
 \item $n=2^r$ with $r\ge2$, and $m=2h(n-1)+2$;
 \item any even $n\ge4$, and $m=h\binom n{n/2}+2$.
\end{enumerate}
In both families, the maximum rooted count over Eulerian orientations is
\[
 \max_D a(D)=
 \frac{n^{m-1}(nm-2)^{n-2}(nm+2n-4)}
      {2^{n+m-1}(n-1)^{n-1}},
\]
and the maximizing orientations are precisely those satisfying
\eqref{eq:eulerian-bipartite-two-level}.
In particular, family \textup{(i)} with $h=1$ determines the maximum on
$K_{n,2n}$ for every $n=2^r\ge4$.
\end{corollary}
\end{samepage}
The proof is given in \Cref{app:bipartite-eulerian-families}.

﻿\section{Concluding remarks and a tournament conjecture}\label{sec:problem53}\label{sec:conclusion}

The regularity theorem reduces odd-order tournament maximization to a problem about the overlap of out-neighborhoods. The degrees of a maximizing tournament are fixed by \cref{thm:regular-main}; the remaining question is which arrangement of its arcs gives the largest rooted count. We conclude with a necessary-condition conjecture motivated by doubly regular tournaments, an exact characteristic-coefficient bound, and the extremizers at small orders.

\subsection{The doubly regular model}
Doubly regular tournaments provide the equality model for the known determinant bound. By Bandekar et al.~\cite[Theorem~1.11]{bandekar2025}, every regular tournament of odd order $n>1$ satisfies
\[
 a(T)\leq\frac1n\left(\frac{n(n+1)}4\right)^{(n-1)/2},
\]
with equality precisely when every pair of distinct vertices has $(n-3)/4$ common outneighbors. These are the doubly regular tournaments. Their orders are $3\pmod4$, and they correspond to skew Hadamard matrices by Reid and Brown~\cite{reidbrown1972}; their spectral characterization also appears in de Caen et al.~\cite{decaen1992}.

The Hadamard property appears here as a uniform Gram matrix. If $A$ is the adjacency matrix of a doubly regular tournament and $M=A-A^{\mathsf T}$, then
\[
 AA^{\mathsf T}=\frac{(n+1)I+(n-3)J}{4},\qquad
 MM^{\mathsf T}=nI-J.
\]
Thus $MM^{\mathsf T}$ acts as the scalar $n$ on $\mathbf1^\perp$: all nonzero squared singular values of $M$ are equal. This uniformity is exactly the equality structure in the arborescence bound.

At order $n=4k+1$, entrywise integrality suggests the corresponding nearly uniform structure. Problem~5.3 of~\cite{bandekar2025} asks for the maximizing Eulerian tournaments in this residue class. Throughout the rest of the section, $T$ is a regular tournament of order $n=4k+1$, with $k\geq1$. Equivalently, $T$ is Eulerian, every outdegree is $2k$, and $\tau(T)=na(T)$. Let $A$ be its $0$--$1$ adjacency matrix, let $M=A-A^{\mathsf T}$, and put
\[
 H_n=\frac{(n+1)I+(n-3)J}{4},\qquad
 B=2(AA^{\mathsf T}-H_n).
\]
The off-diagonal entries of $AA^{\mathsf T}$ are the integers
$c_{ij}=|N_T^+(i)\cap N_T^+(j)|$, whereas those of the doubly regular model $H_n$ are $k-\tfrac12$. Their nearest integer values are $k-1$ and $k$.

We call $T$ \emph{nearly doubly regular}, abbreviated NDR, if these nearest values occur at every pair:
\begin{equation}\label{eq:53-ndr-definition}
 c_{ij}\in\{k-1,k\}\qquad(i\ne j).
\end{equation}
The average common-out-neighbor count is
\[
 \frac{n\binom{2k}{2}}{\binom n2}=k-\frac12,
\]
since the $2k$ inneighbors of each vertex give $\binom{2k}{2}$ unordered pairs having that vertex as a common outneighbor. Hence \eqref{eq:53-ndr-definition} asks that the counts differ by at most one. It is the closest entrywise analogue of double regularity allowed at this order. Near-homogeneous tournaments, an equivalent description in this setting, go back to Tabib~\cite{tabib1980}. Further constructions are given by Tang, Chen and Zhang~\cite{tang2026}, while Satake~\cite{satake2025} studies cyclotomic constructions and their spectra.

\subsection{The fourth coefficient and the proposed structure}
The fourth characteristic coefficient gives an exact measure of this departure from the doubly regular model. For $0\le j\le n$, let $a_j(T)=[x^{n-j}]\Phi_M(x)$, so $a_0(T)=1$. Since $M^{\mathsf T}=-M$, transposition gives $\det(xI-M)=\det(xI+M)$ and hence $\Phi_M(-x)=(-1)^n\Phi_M(x)$. Comparing coefficients shows that $a_j(T)=0$ for every odd index $j$. Moreover, $a_2(T)$ is the sum of the principal $2\times2$ minors of $M$. Each such minor has determinant $1$, since $M_{ij}=-M_{ji}\in\{-1,1\}$ for $i\ne j$. Thus $a_2(T)=\binom n2$, and
\[
 \Phi_M(x)=\det(xI-M)
 =x^n+a_2(T)x^{n-2}+a_4(T)x^{n-4}+\cdots.
\]
By \eqref{eq:regular-arborescence}, the rooted count is obtained by evaluating $\Phi_M$ at $n$. Since the leading coefficient and $a_2(T)$ are fixed by the order and all odd-indexed coefficients are zero, $a_4(T)$ is the first coefficient not fixed by these identities. This motivates studying its maximum and the structure of the tournaments attaining it. Related four-vertex information is studied by Linial and Morgenstern~\cite{linial2016} through directed four-cycles, and by Belkouche et al.~\cite{belkouche2019} through skew adjacency matrices.

The following proposition connects that coefficient with the nearest-integer Gram structure. As in the preliminaries, $\|X\|_F^2=\sum_{i,j}|X_{ij}|^2$ is the squared Frobenius norm.

\begin{proposition}\label{prop:53-four-equivalences}
Let $T$ be a regular tournament of order $n=4k+1$, with $k\geq1$, and let
\[
 \mathcal C_n=\frac{n(n-1)(n^2-3n-8)}8.
\]
Then
\[
 a_4(T)=\frac{n^2(n-1)(n-3)}8
             -4\|AA^{\mathsf T}-H_n\|_F^2,
\]
and
\[
 a_4(T)\leq\mathcal C_n,\qquad
 \|AA^{\mathsf T}-H_n\|_F^2\geq\frac{n(n-1)}4.
\]
The following conditions are equivalent:
\begin{enumerate}
\item $T$ is nearly doubly regular;
\item $a_4(T)=\mathcal C_n$;
\item $AA^{\mathsf T}=H_n+B/2$, where $B$ is symmetric, $B_{ii}=0$, and $B_{ij}\in\{-1,1\}$ for $i\ne j$;
\item $\|AA^{\mathsf T}-H_n\|_F^2=n(n-1)/4$.
\end{enumerate}
\end{proposition}

The proof is given in \cref{app:ndr}. The coefficient identity makes the structural analogy quantitative: maximizing $a_4$ is exactly the same as minimizing the Frobenius distance from $AA^{\mathsf T}$ to $H_n$. NDR tournaments attain the distance bound forced by integer common-neighbor counts. Whenever such a tournament exists at the prescribed order, the four conditions therefore characterize the fourth-coefficient maximizers.

The extremizers at small orders support this proposed analogue of the doubly regular model. Computer-assisted enumeration determines the exact maximum arborescence counts at orders $5,9,13,17$, the nontrivial cases of Problem~5.3 below order~$21$. Every maximizing tournament in these computations is NDR. The programs use exact arithmetic and complete covers of the relevant tournament classes; the programs and certificates are available in the accompanying GitHub repository~\cite{coraxeus2026repository}.

Together, the doubly regular equality model, its nearest-integer counterpart, the fourth-coefficient identity and these finite extremizers motivate the following conjecture.

\begin{conjecture}\label{conj:ndr}
For every $k\geq1$, every regular tournament of order $4k+1$ that maximizes the number of arborescences is nearly doubly regular.
\end{conjecture}

By \cref{prop:53-four-equivalences}, each of its four conditions is an equivalent necessary conclusion in \cref{conj:ndr}. The coefficient formulation specifies equality with $\mathcal C_n$ at every order. Its equivalence with maximizing the actual coefficient over regular tournaments holds at orders admitting an NDR tournament. The remaining characteristic coefficients give the next comparisons within the NDR class, and \eqref{eq:regular-arborescence} expresses their combined contribution to the arborescence count.

\section{AI Declaration}\label{sec:AIdeclaration}

ChatGPT (GPT) was used to assist with the exploration of proof strategies and the formulation and testing of intermediate claims. The authors directed the research, selected the methods, organized the arguments, independently checked all proofs, and take full responsibility for the mathematical content and presentation.

\makeatletter
\def\@biblabel#1{\@defaultbiblabelstyle{#1}}
\def\bibsetup{}
\makeatother
\input{main.bbl}

\appendix

\section{Auxiliary proofs for complete bipartite graphs}\label[appendix]{app:bipartite}
We use the notation of \Cref{sec:problem55}.

\subsection{The spectral product bound}\label[appendix]{app:bipartite-moment-product}
\begin{proof}[Proof of \Cref{lem:eulerian-moment-product}]
Since $z_i\ge0$ and $\sum_i z_i=1$, we have
$\sum_i z_i^2\le(\sum_i z_i)^2=1$ and
$0\le\nu=\sum_i z_i^2-1/p\le(p-1)/p$.
Let $\alpha=\alpha_p(\nu)$ and $\beta=\beta_p(\nu)$.
Then $0\le\alpha\le\beta\le1$. For $\nu=0$, all $z_i=1/p$
and equality holds; assume henceforth $\nu>0$, so $\alpha<\beta$.

For each $j$, the other centered coordinates sum to $-(z_j-1/p)$.
By Cauchy--Schwarz,
\[
 \sum_{i\ne j}(z_i-1/p)^2
 \ge\frac1{p-1}\left(\sum_{i\ne j}(z_i-1/p)\right)^2
 =\frac{(z_j-1/p)^2}{p-1}.
\]
Adding the $j$th square yields
$\nu\ge p(z_j-1/p)^2/(p-1)$, and hence
$z_j\le1/p+\sqrt{(p-1)\nu/p}=\beta$.

Let $f(x)=\log(1+x)$ and let $q$ be the quadratic polynomial satisfying
\[
 q(\alpha)=f(\alpha),\qquad q'(\alpha)=f'(\alpha),\qquad
 q(\beta)=f(\beta).
\]
The first two conditions fix the constant and linear terms about
$\alpha$; the third fixes the quadratic term since $\beta\ne\alpha$.
We show that $q\ge f$ on $[0,\beta]$, with equality at the two nodes.

For $x\in[0,\beta]\setminus\{\alpha,\beta\}$, let
\[
 c=\frac{f(x)-q(x)}{(x-\alpha)^2(x-\beta)},\qquad
 F(y)=f(y)-q(y)-c(y-\alpha)^2(y-\beta).
\]
By construction, $F(x)=F(\beta)=F(\alpha)=F'(\alpha)=0$.
Three applications of Rolle's theorem give $F'''(\xi)=0$ between
the extreme nodes. Since $q'''=0$, we have $c=f'''(\xi)/6$, hence
\[
 f(x)-q(x)=\frac{f'''(\xi)}6(x-\alpha)^2(x-\beta)<0.
\]
Here $f'''(\xi)=2/(1+\xi)^3>0$ and $x<\beta$, $x\ne\alpha$.
Equality holds at the nodes by definition.

Let $s=\sqrt{\nu/[p(p-1)]}$. Then
$\alpha=1/p-s$ and $\beta=1/p+(p-1)s$, so
\begin{align*}
 (p-1)\alpha+\beta&=1,\\
 (p-1)\alpha^2+\beta^2
 &=\frac1p+p(p-1)s^2=\frac1p+\nu.
\end{align*}
Thus $(z_i)$ and $(\alpha,\ldots,\alpha,\beta)$ have the same
first and second moments, and hence the same sum of values of $q$.
The inequality $f\le q$ gives
\[
 \sum_{i=1}^p\log(1+z_i)
 \le\sum_{i=1}^p q(z_i)
 =(p-1)q(\alpha)+q(\beta)
 =\log\mathcal E_p(\nu).
\]
Exponentiating proves the bound. Equality requires every $z_i$ to
be a node. If $\ell$ entries equal $\beta$, their sum is
$p\alpha+\ell(\beta-\alpha)=1=p\alpha+(\beta-\alpha)$,
so $(\ell-1)(\beta-\alpha)=0$ and $\ell=1$.

For monotonicity, let $h=1+1/p$. Then
$\mathcal E_p(\nu)=(h-s)^{p-1}(h+(p-1)s)$ and
$\nu=p(p-1)s^2$. For $s>0$,
\begin{align*}
 \frac{d}{ds}\log\mathcal E_p(\nu)
 &=-\frac{p-1}{h-s}+\frac{p-1}{h+(p-1)s}\\
 &=-\frac{p(p-1)s}{(h-s)(h+(p-1)s)}.
\end{align*}
Dividing by $d\nu/ds=2p(p-1)s$ yields
\[
 \frac{d}{d\nu}\log\mathcal E_p(\nu)
 =-\frac1{2(h-s)(h+(p-1)s)}<0.
\]
Both factors in the denominator are positive. Together with
continuity at zero, this proves strict decrease on $[0,(p-1)/p]$.
\end{proof}

\subsection{Two Eulerian sharpness families}\label[appendix]{app:bipartite-eulerian-families}
\begin{proof}[Proof of \Cref{cor:eulerian-sharp-families}]
For each family we construct a sign matrix $X$ with zero row and
column sums and Gram matrix \eqref{eq:eulerian-bipartite-two-level}.
Then $C=(X+J_{n,m})/2$ has row sums $m/2$ and column sums $n/2$,
and defines an Eulerian orientation attaining
\Cref{thm:eulerian-bipartite-bound}.

For \textup{(i)}, let $n=2^r$ and index coordinates by $x\in\FF_2^r$.
For nonzero $a\in\FF_2^r$, let $v_a=((-1)^{a\cdot x})_x$,
with dot product in $\FF_2$. Choose $x_0$ with $a\cdot x_0=1$.
Translation by $x_0$ pairs opposite terms in the coordinate sum,
so $v_a^{\T}\one=0$. For $a\ne b$, the product character has
nonzero index $a+b$ and also sums to zero. Thus
\[
 v_a^{\T}v_b=\begin{cases}n,&a=b,\\0,&a\ne b.\end{cases}
\]
The $n-1$ vectors $v_a/\sqrt n$ form an orthonormal basis of
$\one^\perp$, so summing their projections gives
\[
 \sum_{a\ne0}v_av_a^{\T}=nP.
\]

Fix $a_0\ne0$. Take $h$ copies of each pair $v_a,-v_a$, followed
by one extra pair $v_{a_0},-v_{a_0}$. These $m=2h(n-1)+2$ columns
have zero column sums, and pairing makes each row sum zero. Their Gram matrix is
\[
 G=2h\sum_{a\ne0}v_av_a^{\T}+2v_{a_0}v_{a_0}^{\T}
  =2hnP+2v_{a_0}v_{a_0}^{\T}
  =\frac{n(m-2)}{n-1}P+2v_{a_0}v_{a_0}^{\T}.
\]
Moreover,
\[
 \lambda=\frac{m(n-2)}{4(n-1)}
 =\frac{h(n-2)}2+\frac{n-2}{2(n-1)}.
\]
The first summand is an integer, so $t=(n-2)/(2(n-1))>0$.
Thus \Cref{thm:eulerian-bipartite-bound} gives attainment with $B_E<U_E$.

The construction also applies to any real Hadamard matrix of order
$n\ge4$: multiplying rows by signs makes one column $\one$, and
the remaining columns are balanced orthogonal sign vectors replacing the $v_a$.

For \textup{(ii)}, let $n\ge4$ be even and $B=\binom n{n/2}$.
Complementation pairs the $B$ balanced sign vectors $y$ with $-y$,
so $B$ is even and their sum is zero. Let
\[
 S=\sum_{y\in\{1,-1\}^n,\,y^{\T}\one=0}yy^{\T}.
\]
The diagonal entries of $S$ are $B$, and coordinate symmetry makes
all off-diagonal entries equal, say to $c$. Since $S\one=0$,
$B+(n-1)c=0$, giving
\[
 S=\frac{nB}{n-1}I_n-\frac B{n-1}J_n=\frac{nB}{n-1}P.
\]

Take $h$ copies of every balanced sign vector and one extra pair $v,-v$.
The resulting $X$ has zero row and column sums, with
\[
 m=hB+2,\qquad
 G=hS+2vv^{\T}=\frac{nhB}{n-1}P+2vv^{\T}
             =\frac{n(m-2)}{n-1}P+2vv^{\T}.
\]
To check the required congruence, let $n=2k$. Then
\[
 B=\binom{2k}{k}
   =2(2k-1)\frac1k\binom{2k-2}{k-1}.
\]
The last factor is integral since
\[
 \frac1k\binom{2k-2}{k-1}
 =\binom{2k-2}{k-1}-\binom{2k-2}{k-2}.
\]
Thus $2(n-1)\mid B$ and $m\equiv2\pmod{2(n-1)}$, with $m$ even
and $m\ge2(n-1)+2\ge n$. Writing $m=2\ell(n-1)+2$ gives
\[
 \lambda=\frac{\ell(n-2)}2+\frac{n-2}{2(n-1)},
 \qquad t=\frac{n-2}{2(n-1)}>0.
\]
The Gram identity and \Cref{thm:eulerian-bipartite-bound} give attainment.

In both constructions, $Pv=v$ and $v^{\T}v=n$. The Gram eigenvalues
are zero on the span of $\one$,
\[
 \frac{n(m-2)}{n-1}+2n=\frac{n(m+2n-4)}{n-1}
\]
on the span of $v$, and $n(m-2)/(n-1)$ with multiplicity $n-2$
on $\{\one,v\}^{\perp}$. Substitution in
\eqref{eq:eulerian-bipartite-determinant} gives
\begin{align*}
 a(D)
 &=A_{n,m}\left(\frac{nm-2}{m(n-1)}\right)^{n-2}
           \frac{nm+2n-4}{m(n-1)}\\
 &=\frac{n^{m-1}(nm-2)^{n-2}(nm+2n-4)}
         {2^{n+m-1}(n-1)^{n-1}}.
\end{align*}
The equality condition in \Cref{thm:eulerian-bipartite-bound}
characterizes all maximizers. In \textup{(i)}, $h=1$ gives
$m=2(n-1)+2=2n$.
\end{proof}

\section{Equivalent conditions for nearly doubly regular tournaments}\label[appendix]{app:ndr}

We prove \cref{prop:53-four-equivalences} in the notation of \cref{sec:problem53}: $T$ is regular of order $n=4k+1$, $A$ is its adjacency matrix, $M=A-A^{\mathsf T}$, and
\[
 H_n=\frac{(n+1)I+(n-3)J}{4},\qquad
 B=2(AA^{\mathsf T}-H_n).
\]
Traces and Frobenius norms refer to the full $n\times n$ matrices.

\begin{proof}[Proof of \cref{prop:53-four-equivalences}]
Regularity, with $2k=(n-1)/2$, and $A+A^{\mathsf T}=J-I$ give
\begin{gather*}
 A\mathbf1=A^{\mathsf T}\mathbf1=2k\mathbf1,
 \qquad M\mathbf1=M^{\mathsf T}\mathbf1=0,\\
 2A=M+J-I,\qquad 2A^{\mathsf T}=-M+J-I.
\end{gather*}
Since $M\mathbf1=0$ and $\mathbf1^{\mathsf T}M=0$, we have $MJ=JM=0$; in particular, $M$ commutes with $J-I$. Multiplying the two displayed expressions yields
\begin{align*}
 4AA^{\mathsf T}
 &=(M+J-I)(-M+J-I)\\
 &=-M^2+M(J-I)-(J-I)M+(J-I)^2\\
 &=-M^2+I+(n-2)J.
\end{align*}
The last equality uses $J^2=nJ$. Put $R=MM^{\mathsf T}=-M^2$. Substituting this product formula into the definition of $B$ gives
\[
 B=\frac{R+I+(n-2)J-(n+1)I-(n-3)J}{2}
   =\frac{R-nI+J}{2}.
\]

The diagonal of $AA^{\mathsf T}$ is $2k$, while its $(i,j)$ entry for $i\ne j$ is $c_{ij}$. The diagonal of $H_n$ is also $2k$, and its off-diagonal entries are $(n-3)/4=k-\tfrac12$. Consequently,
\begin{equation}\label{eq:53-B-identities}
 \begin{aligned}
 B&=\frac{MM^{\mathsf T}-nI+J}{2},\qquad B\mathbf1=0,\\
 B_{ii}&=0,\qquad B_{ij}=2c_{ij}-(2k-1)\quad(i\ne j).
 \end{aligned}
\end{equation}
Here $B\mathbf1=0$ follows from $R\mathbf1=0$ and $J\mathbf1=n\mathbf1$; symmetry follows from that of $AA^{\mathsf T}$ and $H_n$.

The off-diagonal entries of $B$ are odd by \eqref{eq:53-B-identities}. For odd $z$, we have $z^2\geq1$ and $z^2\equiv1\pmod8$, so
\[
 \delta(T):=\sum_{i<j}\frac{B_{ij}^2-1}{8}
\]
is a nonnegative integer. Symmetry and the zero diagonal of $B$ give
\begin{align*}
 \operatorname{tr}(B^2)
 &=\sum_{i,j}B_{ij}B_{ji}
  =2\sum_{i<j}B_{ij}^2\\
 &=2\left(\binom n2+8\delta(T)\right)
  =n(n-1)+16\delta(T).
\end{align*}
Using $AA^{\mathsf T}-H_n=B/2$ gives
\[
 \|AA^{\mathsf T}-H_n\|_F^2
 =\frac14\operatorname{tr}(B^2)
 =\frac{n(n-1)}4+4\delta(T).
\]
Each term defining $\delta(T)$ is nonnegative, so $\delta(T)=0$ exactly when every off-diagonal entry is a sign. By \eqref{eq:53-B-identities},
\[
 B_{ij}=-1\ \Longleftrightarrow\ c_{ij}=k-1,
 \qquad
 B_{ij}=1\ \Longleftrightarrow\ c_{ij}=k.
\]
Thus $\delta(T)=0$ is equivalent to the NDR condition \eqref{eq:53-ndr-definition}.

To relate $\|AA^{\mathsf T}-H_n\|_F^2$ to $a_4(T)$, write
\[
 \Phi_M(x)=\sum_{j=0}^n\gamma_jx^{n-j},\qquad
 \gamma_0=1,\qquad p_j=\operatorname{tr}(M^j).
\]
Transposition gives $\det(xI-M)=\det(xI+M)$, so oddness of $n$ implies $\Phi_M(-x)=-\Phi_M(x)$. Thus $\gamma_1=\gamma_3=0$, $\gamma_2=a_2(T)$ and $\gamma_4=a_4(T)$. Every odd power of $M$ is skew-symmetric, giving $p_1=p_3=0$.

For these coefficients, Newton's identities take the form
\[
 p_j+\gamma_1p_{j-1}+\cdots+\gamma_{j-1}p_1+j\gamma_j=0
 \qquad(1\leq j\leq4).
\]
At $j=2$ and $j=4$, the equalities $\gamma_1=\gamma_3=p_1=p_3=0$ reduce them to
\[
 p_2+2a_2(T)=0,\qquad
 p_4+a_2(T)p_2+4a_4(T)=0.
\]
Each off-diagonal entry of $M$ is $1$ or $-1$, and $M_{ji}=-M_{ij}$. Thus
\[
 p_2=\sum_i\sum_{j\ne i}M_{ij}M_{ji}=-n(n-1).
\]
The $j=2$ identity gives $a_2(T)=n(n-1)/2$, and the $j=4$ identity then gives
\[
 a_4(T)=-\frac{p_4}{4}+\frac{p_2^2}{8}
       =\frac{n^2(n-1)^2}{8}-\frac{\operatorname{tr}(R^2)}4,
\]
where the last equality uses $R=-M^2$ and hence $R^2=M^4$.

By \eqref{eq:53-B-identities}, $R=nI-J+2B$. Expanding its square under the trace gives
\[
 \operatorname{tr}(R^2)
 =\operatorname{tr}((nI-J)^2)
  +4\operatorname{tr}((nI-J)B)+4\operatorname{tr}(B^2).
\]
Since $(nI-J)^2=n^2I-nJ$, the first trace is $n^3-n^2=n^2(n-1)$. The zero diagonal and row sums of $B$ give the cross term
\[
 \operatorname{tr}((nI-J)B)
 =n\operatorname{tr}(B)-\operatorname{tr}(JB)
 =0-\mathbf1^{\mathsf T}B\mathbf1=0.
\]
Therefore
\[
 \operatorname{tr}(R^2)=n^2(n-1)+4\operatorname{tr}(B^2).
\]
Substituting this trace into the preceding formula for $a_4(T)$, using $B=2(AA^{\mathsf T}-H_n)$, yields
\begin{align*}
 a_4(T)
 &=\frac{n^2(n-1)^2}{8}-\frac{n^2(n-1)}4
     -\operatorname{tr}(B^2)\\
 &=\frac{n^2(n-1)(n-3)}8
     -4\|AA^{\mathsf T}-H_n\|_F^2.
\end{align*}
Substituting the preceding formula for $\|AA^{\mathsf T}-H_n\|_F^2$ gives
\begin{align*}
 a_4(T)
 &=\frac{n^2(n-1)(n-3)}8-n(n-1)-16\delta(T)\\
 &=\frac{n(n-1)(n^2-3n-8)}8-16\delta(T)
  =\mathcal C_n-16\delta(T).
\end{align*}

The formulas for $a_4(T)$ and $\|AA^{\mathsf T}-H_n\|_F^2$ give both bounds in \cref{prop:53-four-equivalences}, since $\delta(T)\geq0$. They also show that conditions~(2) and~(4) each hold exactly when $\delta(T)=0$, which is equivalent to condition~(1). Since $B$ is symmetric with zero diagonal, condition~(3) requires precisely $B_{ij}\in\{-1,1\}$ for $i\ne j$, again equivalent to $\delta(T)=0$. Thus all four conditions are equivalent.
\end{proof}

\end{document}

%% file: main.bbl
\providecommand{\bysame}{\leavevmode\hbox to3em{\hrulefill}\thinspace}
\providecommand{\MR}{\relax\ifhmode\unskip\space\fi MR }
\providecommand{\MRhref}[2]{%
  \href{http://www.ams.org/mathscinet-getitem?mr=#1}{#2}
}
\providecommand{\href}[2]{#2}